\documentclass[10pt,reqno]{amsart}

\usepackage[T1]{fontenc}
\usepackage[utf8x]{inputenc}
\usepackage[french, english]{babel}
\usepackage{amsmath,amsfonts,amssymb,amsthm,bm,shuffle}
\usepackage[math]{anttor}
\usepackage{pgf,tikz}
\usetikzlibrary{arrows}
\usepackage{float} 

\usepackage{mathrsfs}
\usepackage{setspace}
\usepackage{amsfonts}
\usepackage{pb-diagram}

\usepackage{tikz}
\usepackage{tikz-cd} 
\usetikzlibrary{arrows, matrix}

\usepackage[top=3.4cm,bottom=3.4cm,left=3.7cm,right=3.7cm]{geometry}

\usepackage{xcolor}
\usepackage[hyperindex=true,frenchlinks=true,colorlinks=true,
citecolor=Col5,linkcolor=Col6,urlcolor=Col5,linktocpage,
pagebackref=true]{hyperref}

\definecolor{ColBlack}{RGB}{0,0,0} 
\definecolor{ColWhite}{RGB}{255,255,255} 
\definecolor{Col1}{RGB}{199, 0, 57} 
\definecolor{Col2}{RGB}{144, 12, 63} 
\definecolor{Col3}{RGB}{88, 24, 69} 
\definecolor{Col4}{RGB}{23, 1, 202} 
\definecolor{Col5}{RGB}{23, 107, 240} 
\definecolor{Col6}{RGB}{32, 103, 143} 

\usepackage{tikz}
\usetikzlibrary{shapes}
\usetikzlibrary{fit}
\usetikzlibrary{decorations.pathmorphing}
\usetikzlibrary{arrows.meta}

\usepackage{mathtools}
\usepackage{dsfont}
\usepackage{wasysym}
\usepackage{stmaryrd}
\usepackage{nameref}
\usepackage{youngtab}
\usepackage{cite}
\usepackage{xspace}
\usepackage{subfig}
\usepackage{multirow}
\usepackage{enumitem}
\usepackage{multicol}
\usepackage{chngcntr}

\tikzstyle{PathNode}=[circle,draw=Col1,fill=Col1!10,thick,inner sep=0pt,minimum size=1.5mm]
\tikzstyle{PathStep}=[color=Col1!70]

\tikzstyle{Centering}=[{baseline={([yshift=-0.5ex]current
    bounding box.center)}}]

\tikzstyle{MarkAA}=[draw=Col1!80,fill=Col1!8]
\tikzstyle{MarkAB}=[draw=Col2!80,fill=Col2!8]
\tikzstyle{MarkAC}=[draw=Col3!80,fill=Col3!8]
\tikzstyle{MarkBA}=[draw=Col4!80,fill=Col4!8]
\tikzstyle{MarkBB}=[draw=Col5!80,fill=Col5!8]
\tikzstyle{MarkBC}=[draw=Col6!80,fill=Col6!8]

\tikzstyle{Node}=[circle,MarkAA,inner sep=1pt,
minimum size=2mm,thick,font=\scriptsize]
\tikzstyle{Edge}=[draw=Col1!80,cap=round,thick,rounded corners=2.5pt]
\tikzstyle{Leaf}=[rectangle,MarkBC,inner sep=0pt,minimum size=1mm,thick]
\tikzstyle{NodeST}=[font=\scriptsize]

\tikzstyle{NodeGraph}=[circle,MarkAB,inner sep=1pt,
minimum size=1.5mm,thick]
\tikzstyle{NodeLabelGraph}=[font=\scriptsize,node distance=4mm]
\tikzstyle{EdgeGraph}=[Col1!70,cap=round,very thick]
\tikzstyle{EdgeLabel}=[midway,inner sep=1pt,fill=ColWhite!0,
font=\tiny]
\tikzstyle{FaceXY}=[fill=Col1,opacity=.1]
\tikzstyle{FaceXZ}=[fill=Col2,opacity=.2]
\tikzstyle{FaceYZ}=[fill=Col3,opacity=.2]

\tikzstyle{Map}=[ColBlack!100,draw,-{>[scale=1.5,length=4,width=5]}]
\tikzstyle{Injection}=[ColBlack!100,draw,
{>[scale=1.5,length=4,width=5]}-{>[scale=1.5,length=4,width=5]}]

\tikzstyle{LineGrid}=[very thin,dashed,draw=ColBlack!25]
\tikzstyle{Grid}=[LineGrid]

\newcommand{\Hide}[1]{\textcolor{Col4}{\tt [hidden]}}
\newcommand{\Par}[1]{\left(#1\right)}
\newcommand{\Bra}[1]{\left\{#1\right\}}
\newcommand{\Han}[1]{\left[#1\right]}

\newcommand{\Def}[1]{\textcolor{Col1}{\em #1}}
\newcommand{\OEIS}[1]{\href{http://oeis.org/#1}{{\bf #1}}}

\DeclareRobustCommand{\gobblefive}[5]{}

\let\SavedCaption=\caption
\renewcommand*{\caption}[2][\shortcaption]{%
    \def\shortcaption{#2}
    \SavedCaption[\; #1]{#2}}

\let\SavedParagraph=\paragraph
\renewcommand{\paragraph}[1]{%
    \SavedParagraph{\it #1}}

\newcommand*\ClearToLeftPage{%
    \clearpage
    \ifodd\value{page}
        \hbox{}
        \vspace*{\fill}
        \thispagestyle{empty}
        \newpage
    \fi
}

\newcommand{\DrawGridSpace}[3]{
    \foreach \x in {0,...,#1} {\foreach \y in {0,...,#2} 					{\foreach \z in {0,...,#3} {
        \draw[LineGrid](0,\y,\z)--(\x,\y,\z);
        \draw[LineGrid](\x,0,\z)--(\x,\y,\z);
        \draw[LineGrid](\x,\y,0)--(\x,\y,\z);}}}}

\renewcommand{\arraystretch}{1.4}

\numberwithin{equation}{section}

\counterwithin{figure}{section}
\counterwithin{table}{section}

\makeatletter
\def\l@part{\@tocline{-1}{20pt}{0pc}{5pc}{\Large \bf}}
\def\l@section{\@tocline{1}{3pt}{1pc}{5pc}{}}
\def\l@subsection{\@tocline{2}{2pt}{2pc}{5pc}{}}
\makeatother

\renewcommand{\leq}{\leqslant}
\renewcommand{\geq}{\geqslant}

\newcommand{\SetTriword}{\mathsf{Tr}}
\newcommand{\F}{\mathsf{F}}
\newcommand{\SetDyckPath}{\mathsf{Dy}}
\newcommand{\DexterOrder}{\mathsf{Dex}}

\newcommand{\N}{\mathbb{N}}
\newcommand{\Z}{\mathbb{Z}}

\newcommand{\R}{\mathbb{R}}

\newcommand{\K}{\mathbb{K}}

\newcommand{\DiffIndexes}{\mathrm{D}}
\newcommand{\CubicReal}{\mathfrak{C}}

\newcommand{\DegreePolynomial}{\mathrm{d}}
\newcommand{\In}{\mathrm{in}}
\newcommand{\Out}{\mathrm{out}}

\newcommand{\JoinIrreducible}{\mathbf{J}}
\newcommand{\MeetIrreducible}{\mathbf{M}}
\newcommand{\JLattice}{\mathbb{J}}

\DeclareMathOperator{\JJoin}{\vee}
\DeclareMathOperator{\Meet}{\wedge}

\newcommand{\LatticeL}{\mathcal{L}}
\newcommand{\PosetP}{\mathcal{P}}

\newcommand{\Leq}{\preccurlyeq}
\DeclareMathOperator{\Covered}{\lessdot}

\newcommand{\IntervalTwo}{\mathbf{2}}

\newcommand{\NumberTwo}{\mathfrak{t}}
\newcommand{\Spine}{\mathbb{S}}

\newtheorem{Theorem}{Theorem}[section]
\newtheorem{Proposition}[Theorem]{Proposition}
\newtheorem{Lemma}[Theorem]{Lemma}

\newtheorem{conjecture}[Theorem]{Conjecture}

\title[A geometric and combinatorial exploration of Hochschild lattices]{A geometric and combinatorial exploration\\ of Hochschild lattices}
\subjclass[2010]{}
\date{\today}
\author{Camille Combe}
\address{\scriptsize 
Institut de Recherche Mathématique Avancée 
UMR 7501, Université de Strasbourg et CNRS, 
7 rue René Descartes 
67000 Strasbourg, France}
\email{combe@math.unistra.fr}

\begin{document}

\begin{abstract}
Hochschild lattices are specific intervals in the dexter meet-semilattices recently introduced by Chapoton. A natural geometric realization of these lattices leads to some cell complexes introduced by Saneblidze, called the Hochschild polytopes. We obtain several geometrical properties of the Hochschild lattices, namely we give cubic realizations, establish that these lattices are EL-shellable, and show that they are constructible by interval doubling. We also prove several combinatorial properties as the enumeration of their $k$-chains and compute their degree polynomials.
\end{abstract}

\maketitle

\tableofcontents

\section*{Introduction}
In~\cite{Cha20}, Chapoton introduces new meet-semilattices called dexter posets, defined on the set of Dyck paths, endowed with the dexter order. An interesting and surprising link is found in this article: a connection between some specific intervals of dexter posets and cell complexes introduced by Saneblidze~\cite{San09, San11} in the area of algebraic topology. These cell complexes are called  Hochschild polytopes by Saneblidze. They provide, in the context of algebraic topology, combinatorial cellular models of free loops spaces. There are several ways to build Hochschild polytopes. For instance, they can be obtained by a sequence of truncations of the $n$-simplex, where $n$ is the dimension of the polytopes~\cite{RS18}.
\smallbreak

It is shown in~\cite{Cha20} that the set of Dyck paths in these specific intervals in dexter posets is in bijection with a set of words defined on the alphabet $\{0, 1, 2\}$ satisfying some conditions. Better than that, by considering the poset on this set of words endowed with the componentwise order, Chapoton shows that a covering relation on Dyck paths for the dexter order implies by this bijection a covering relation on the corresponding words. 
\smallbreak

As a first contribution of the present work, we show the reverse implication. This implies that the two posets are isomorphic. Moreover, we show that these posets are lattices. Because of their links with cell complexes of Saneblidze, we call these lattices Hochschild lattices. 
Our goal is to present a geometric and combinatorial exploration of Hochschild lattices, revealing several interesting features. To this aim, we shall mainly work with the word version of the lattice previously mentioned, whose elements are called triwords.   
\smallbreak

In the first section, we recall several definitions by starting with the one of the dexter semilattices and the bijection between Dyck paths of the specific intervals and triwords. 
We divide our study of the posets into two strands: a geometric one and a combinatorial one.
Section~\ref{sec:geom-prop} is devoted to the geometric properties. First, we provide a natural geometric realization for Hochschild lattices, by placing triwords of size $n$ in the space $\mathbb{R}^n$ and by linking by an edge triwords which are in a covering relation. Thanks to this realization, called cubic realization, we are able to show that Hochschild lattices are EL-shellable and constructible by interval doubling (or equivalently congruence uniform~\cite{Muh19}). 
Section~\ref{sec:combi-prop} is about enumerative and combinatorial results. We give here for instance the degree polynomial of the Hochschild lattices that enumerates the triwords with respect to their coverings and the elements they cover. We also provide a formula to compute the number of intervals of these lattices, as well as a method to compute the number of $k$-chains. Section~\ref{sec:combi-prop} ends with the introduction of an interesting subposet of the Hochschild poset, which seems to have similar nice properties.
An appendix on Coxeter polynomials written by Chapoton is added at the end of this article.
\smallbreak

\subsubsection*{General notations and conventions}
Throughout this article, for all words $u$, we denote by $u_i$ the $i$-th letter of $u$. For any word $a$ and integer $k$, $a^k$ is the word $a$ repeated $k$ times.
For all integers $i$ and $j$, $[i, j]$ denotes the set $\{i, i + 1, \dots, j\}$. For any integer $i$, $[i]$ denotes the set $[1, i]$. Graded sets are sets decomposing as a disjoint
union $S = \bigsqcup_{n \geq 0} S(n)$. For any $x \in S$, the unique $n \geq 0$ such that $x
\in S(n)$ is the size $|x|$ of $x$. The empty word is denoted by $\epsilon$.
For all matrices $M$, we denote by $M(i,j)$ the entry at the $i$-th line and the $j$-th column.
All sets considered in this article are finite.
\smallbreak

\subsubsection*{Acknowledgements}
The author was supported by the ANR project Combiné ANR-19-CE48-0011. The author would also like to thank Frédéric Chapoton for the subject and for his many advices.
\section{Definitions and first properties}\label{sec:Def-properties}

\subsection{Hochschild polytopes and triwords}\label{subsec:Hoch}
Let $n \geq 0$ and $w = a_1 a_2 \dots a_n$ be a word of size $n$. The \Def{prefixes} of $w$ are the $n + 1$ words $\epsilon$, $a_1 \dots a_i$, and the \Def{suffixes} of $w$ are the $n + 1$ words $\epsilon$, $a_i \dots a_n$, with $i \in [n]$. A word $x$ is a \Def{factor} of $w$ if there is a prefix $p$ and a suffix $s$ such that $w = p x s$. A word $y$ is a \Def{subword} of $w$ if $y$ can be obtained by deleting letters in $w$. For instance, $radar$ is a subword of $abracadabra$. 
\smallbreak

To describe our objects introduced in the sequel, we use the regular expression notation~\cite{Sak09}.
Recall that for a letter $a$, $a^*$ denote the set of words $a^k$ for any $k\in \mathbb{N}$, and $a^+$ denote the set of words $aa^*$ . Besides, for two expressions $r$ and $s$, $r + s$ is the union of the two sets denoted by $r$ and $s$.
\smallbreak

For any $n\geq 0$, a \Def{Dyck path} of size $n$ is a lattice path from $(0,0)$ to $(2n, 0)$ which stays above the horizontal line, and which consists only of north-east steps and south-east steps. The graded set of Dyck paths is denoted by $\SetDyckPath$ where the size of a Dyck path is its number of north-east steps.
To simplify, we see a Dyck path of size $n$ as a binary sequence of length $2n$ where the letter $1$ encodes a north-east step and the letter $0$ encodes a south-east step. In this section, we recall several definitions, concepts, and notations given in~\cite{Cha20}.
\smallbreak

Let $d \in \SetDyckPath(n)$. 
The Dyck path $d$ is \Def{primitive} if for all Dyck paths $x$ and $y$ such that $d = xy$, one has $x = \epsilon$ or $y = \epsilon$.
A factor $x$ is a \Def{subpath} of $d$ if $x$ is a Dyck path.
A subpath $x$ of $d$ is \Def{movable} if $x$ is primitive and if there is a prefix $p$ and a suffix $s$ such that $d = p1 0^{m} x s$, where $m > 0$, and either $s = \epsilon$ or the first letter of $s$ is $1$. Figure~\ref{fig:example movable paths} gives two examples of movable subpaths. 
\smallbreak

\begin{figure}[ht]
    \centering
    \subfloat[][]{
    \centering
    \scalebox{1}{
        \begin{tikzpicture}[Centering,scale=.4]
            \draw[Grid](0,0)grid(10,2);
            \node[PathNode](0)at(0,0){};
            \node[PathNode](1)at(1,1){};
            \node[PathNode](2)at(2,2){};
            \node[PathNode](3)at(3,1){};
            \node[PathNode](4)at(4,0)[circle,draw=Col4,fill=Col4!20,thick,inner sep=0pt,minimum size=1.5mm]{};
            \node[PathNode](5)at(5,1)[circle,draw=Col4,fill=Col4!20,thick,inner sep=0pt,minimum size=1.5mm]{};
            \node[PathNode](6)at(6,0)[circle,draw=Col4,fill=Col4!20,thick,inner sep=0pt,minimum size=1.5mm]{};
            \node[PathNode](7)at(7,1){};
            \node[PathNode](8)at(8,2){};
            \node[PathNode](9)at(9,1){};
            \node[PathNode](10)at(10,0){};
            \draw[PathStep](0)--(1);
            \draw[PathStep](1)--(2);
            \draw[PathStep](2)--(3);
            \draw[PathStep](3)--(4);
            \draw[PathStep](4)--(5)[color=Col4,very thick];
            \draw[PathStep](5)--(6)[color=Col4,very thick];
            \draw[PathStep](6)--(7);
            \draw[PathStep](7)--(8);
            \draw[PathStep](8)--(9);
            \draw[PathStep](9)--(10);
        \end{tikzpicture}}
    \label{subfig:movable-path1}}
    \qquad
    \subfloat[][]{
    \centering
    \scalebox{1}{
		\begin{tikzpicture}[Centering,scale=.4]
            \draw[Grid](0,0)grid(10,2);
            \node[PathNode](0)at(0,0){};
            \node[PathNode](1)at(1,1){};
            \node[PathNode](2)at(2,2){};
            \node[PathNode](3)at(3,1){};
            \node[PathNode](4)at(4,0){};
            \node[PathNode](5)at(5,1){};
            \node[PathNode](6)at(6,0)[circle,draw=Col4,fill=Col4!20,thick,inner sep=0pt,minimum size=1.5mm]{};
            \node[PathNode](7)at(7,1)[circle,draw=Col4,fill=Col4!20,thick,inner sep=0pt,minimum size=1.5mm]{};
            \node[PathNode](8)at(8,2)[circle,draw=Col4,fill=Col4!20,thick,inner sep=0pt,minimum size=1.5mm]{};
            \node[PathNode](9)at(9,1)[circle,draw=Col4,fill=Col4!20,thick,inner sep=0pt,minimum size=1.5mm]{};
            \node[PathNode](10)at(10,0)[circle,draw=Col4,fill=Col4!20,thick,inner sep=0pt,minimum size=1.5mm]{};
            \draw[PathStep](0)--(1);
            \draw[PathStep](1)--(2);
            \draw[PathStep](2)--(3);
            \draw[PathStep](3)--(4);
            \draw[PathStep](4)--(5);
            \draw[PathStep](5)--(6);
            \draw[PathStep](6)--(7)[color=Col4,very thick];
            \draw[PathStep](7)--(8)[color=Col4,very thick];
            \draw[PathStep](8)--(9)[color=Col4,very thick];
            \draw[PathStep](9)--(10)[color=Col4,very thick];
        \end{tikzpicture}}
    \label{subfig:movable-path2}}
  \caption{\footnotesize A Dyck path $1100101100$ with two movable paths, in blue (dark).}
    \label{fig:example movable paths}
\end{figure}
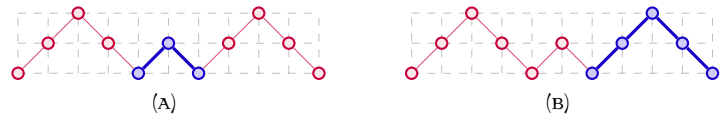

Recall the definition of the dexter order, introduced in~\cite{Cha20}.
For any $n \geq 0$, let $d :=  p 10^{m} x s$ be a Dyck path of size $n$, where $x$ is movable. Let $d_{\alpha,\beta}$ be the Dyck path of size $n$ such that $d_{\alpha,\beta} := p1 0^{\alpha} x 0^{\beta} s$, where $\alpha + \beta = m$ and $\beta >0$. We set that $d \Covered_{\DexterOrder} d'$ if and only if $d' = d_{\alpha,\beta}$, for any $x$ movable subpath of $d$. The dexter order, denoted by $\Leq_{\DexterOrder}$, is the reflexive and transitive closure of $\Covered_{\DexterOrder}$, which is the covering relation. Figure~\ref{fig:example covering dexter} depicts the three covering Dyck paths of the Dyck path $1100101100$ seen in Figure~\ref{fig:example movable paths} for the dexter order. Note that the chosen movable subpath $x$ is no longer movable in $d_{\alpha,\beta}$. 
\smallbreak

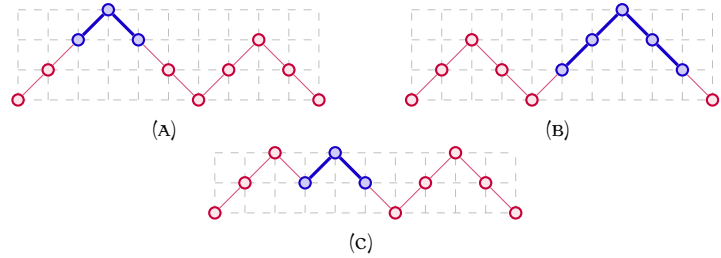
\begin{figure}[ht]
    \centering
    \subfloat[][]{
    \centering
    \scalebox{1}{
    \begin{tikzpicture}[Centering,scale=.4]
            \draw[Grid](0,0)grid(10,3);
            \node[PathNode](0)at(0,0){};
            \node[PathNode](1)at(1,1){};
            \node[PathNode](2)at(2,2)[circle,draw=Col4,fill=Col4!20,thick,inner sep=0pt,minimum size=1.5mm]{};
            \node[PathNode](3)at(3,3)[circle,draw=Col4,fill=Col4!20,thick,inner sep=0pt,minimum size=1.5mm]{};
            \node[PathNode](4)at(4,2)[circle,draw=Col4,fill=Col4!20,thick,inner sep=0pt,minimum size=1.5mm]{};
            \node[PathNode](5)at(5,1){};
            \node[PathNode](6)at(6,0){};
            \node[PathNode](7)at(7,1){};
            \node[PathNode](8)at(8,2){};
            \node[PathNode](9)at(9,1){};
            \node[PathNode](10)at(10,0){};
            \draw[PathStep](0)--(1);
            \draw[PathStep](1)--(2);
            \draw[PathStep](2)--(3)[color=Col4,very thick];
            \draw[PathStep](3)--(4)[color=Col4,very thick];
            \draw[PathStep](4)--(5);
            \draw[PathStep](5)--(6);
            \draw[PathStep](6)--(7);
            \draw[PathStep](7)--(8);
            \draw[PathStep](8)--(9);
            \draw[PathStep](9)--(10);
        \end{tikzpicture}}
    \label{subfig:cover1}}
    \qquad
    \subfloat[][]{
    \centering
    \scalebox{1}{
    \begin{tikzpicture}[Centering,scale=.4]
            \draw[Grid](0,0)grid(10,3);
            \node[PathNode](0)at(0,0){};
            \node[PathNode](1)at(1,1){};
            \node[PathNode](2)at(2,2){};
            \node[PathNode](3)at(3,1){};
            \node[PathNode](4)at(4,0){};
            \node[PathNode](5)at(5,1)[circle,draw=Col4,fill=Col4!20,thick,inner sep=0pt,minimum size=1.5mm]{};
            \node[PathNode](6)at(6,2)[circle,draw=Col4,fill=Col4!20,thick,inner sep=0pt,minimum size=1.5mm]{};
            \node[PathNode](7)at(7,3)[circle,draw=Col4,fill=Col4!20,thick,inner sep=0pt,minimum size=1.5mm]{};
            \node[PathNode](8)at(8,2)[circle,draw=Col4,fill=Col4!20,thick,inner sep=0pt,minimum size=1.5mm]{};
            \node[PathNode](9)at(9,1)[circle,draw=Col4,fill=Col4!20,thick,inner sep=0pt,minimum size=1.5mm]{};
            \node[PathNode](10)at(10,0){};
            \draw[PathStep](0)--(1);
            \draw[PathStep](1)--(2);
            \draw[PathStep](2)--(3);
            \draw[PathStep](3)--(4);
            \draw[PathStep](4)--(5);
            \draw[PathStep](5)--(6)[color=Col4,very thick];
            \draw[PathStep](6)--(7)[color=Col4,very thick];
            \draw[PathStep](7)--(8)[color=Col4,very thick];
            \draw[PathStep](8)--(9)[color=Col4,very thick];
            \draw[PathStep](9)--(10);
        \end{tikzpicture}}
    \label{subfig:cover3}}
    \qquad
    \subfloat[][]{
    \centering
    \scalebox{1}{
    \begin{tikzpicture}[Centering,scale=.4]
            \draw[Grid](0,0)grid(10,2);
            \node[PathNode](0)at(0,0){};
            \node[PathNode](1)at(1,1){};
            \node[PathNode](2)at(2,2){};
            \node[PathNode](3)at(3,1)[circle,draw=Col4,fill=Col4!20,thick,inner sep=0pt,minimum size=1.5mm]{};
            \node[PathNode](4)at(4,2)[circle,draw=Col4,fill=Col4!20,thick,inner sep=0pt,minimum size=1.5mm]{};
            \node[PathNode](5)at(5,1)[circle,draw=Col4,fill=Col4!20,thick,inner sep=0pt,minimum size=1.5mm]{};
            \node[PathNode](6)at(6,0){};
            \node[PathNode](7)at(7,1){};
            \node[PathNode](8)at(8,2){};
            \node[PathNode](9)at(9,1){};
            \node[PathNode](10)at(10,0){};
            \draw[PathStep](0)--(1);
            \draw[PathStep](1)--(2);
            \draw[PathStep](2)--(3);
            \draw[PathStep](3)--(4)[color=Col4,very thick];
            \draw[PathStep](4)--(5)[color=Col4,very thick];
            \draw[PathStep](5)--(6);
            \draw[PathStep](6)--(7);
            \draw[PathStep](7)--(8);
            \draw[PathStep](8)--(9);
            \draw[PathStep](9)--(10);
        \end{tikzpicture}}
    \label{subfig:cover2}}
  \caption{\footnotesize The three Dyck paths covering the Dyck path $1100101100$ for the dexter order.}
    \label{fig:example covering dexter}
\end{figure}

The set $\SetDyckPath(n)$ endowed with the dexter order is a meet-semilattice with many properties highlighted in the article of Chapoton. In this article, we restrict ourselves to a particular interval of this semilattice. 
\smallbreak

In any Dyck path $d$, a factor $01$ is called a \Def{valley}. The \Def{height} of a valley is the ordinate of its corresponding middle point in the path.
\smallbreak

For any $n \geq 1$, let $\F(n)$ be the interval in $\SetDyckPath(n+2)$ between $1100(10)^{n}$ and $11^{n}0^{n}100$.
In particular, any $d$ in the interval $\F(n)$ satisfies the three following assertions:
\begin{itemize}
\item the sequence of heights of the valleys in $d$ is weakly decreasing from left to right,
\item the Dyck path $d$ ends either with $010$ or $0100$,
\item the Dyck path $d$ starts with $11$ and has only valleys of height $0$ or $1$.
\end{itemize}

For any $n \geq 1$, let us recall the bijection $\rho$ between $\F(n)$ and the set of words of length $n$ in the alphabet $\{0,1,2\}$ satisfying some conditions. Let $d \in \F(n)$ and $N_2$ be an integer initially set to $0$. 
By reading from left to right the word $d$, let us build the word $u$, initially the empty word, by following the two conditions,
\begin{enumerate}[label=(\roman*)]
\item when two consecutive $1$ are read in $d$, except the first two letters of $d$, then $1$ is added to $N_2$,
\item when a valley of height $h$ is read in $d$, the word $h 2^{N_2}$ is added at the end of the building word $u$, and $N_2$ is then set back to $0$\footnote{The word $2^{N_2}$ means that the letter $2$ is repeated $N_2$ times.}.
\end{enumerate}
The result $\rho(d)$ is the word $u$ obtained after reading all $d$. The length of $u$ is $n$ because, except the two initial letters $1$, every letter $1$ in $d$ contributes a letter in $u$.
\smallbreak

For instance, the image by $\rho$ of the two Dyck paths $1101001010$ and $1110010010$, both in $\F(3)$, are respectively $100$ and $120$.
\smallbreak

Since we are going to work in this article on the set $\rho(\F(n))$, we need to give a description of this set which is independent of the construction induced by $\rho$.
\smallbreak

For any $n \geq 1$, a word $u$ of size $n$ is a \Def{triword} of the same size if $u$ satisfies, for all $i \in [n]$,
\begin{enumerate}[label=(\roman*)]
\item $u_i \in \{0,1,2\}$,
\item $u_1 \ne 2$,
\item\label{subword01} if $u_i = 0$ then $u_j \ne 1$ for all $j > i$.
\end{enumerate}
The graded set of triwords is denoted by $\SetTriword$, where the size of a triword is its number of letters. 
\smallbreak

For instance, 
\begin{equation}
\begin{split}
\SetTriword(1) &= \{0, 1\}, \qquad \SetTriword(2) = \{00, 02, 10, 11, 12\}, \\
\SetTriword(3) &= \{000, 020, 002, 100, 022, 110, 102, 120, 111, 121, 112, 122\}.
\end{split}
\end{equation}
\smallbreak

Note that the condition~\ref{subword01} means that there is no subword $01$ in any triword. 
\smallbreak

\begin{Lemma}\label{grammar}
The set of triwords is specified by the formal grammar
\begin{align}
A &= \epsilon + 0A + 2A, \label{grammar1}\\
B &= \epsilon + 0A + 1B + 2B, \label{grammar2}\\
\SetTriword &= \epsilon + 0A + 1B. \label{grammar3}
\end{align}
\end{Lemma}

\begin{proof}
First, $A$ is the set of all words on ${0,2}$. 
By induction on the length of the words, one can prove that $B$ is the set of all words on $\{0,1,2\}$ avoiding the subword $01$.
Finally, since a triword beginning by $0$ has no occurrences of $1$, and a triword beginning by $1$ writes as $1u'$ where $u' \in B$, \eqref{grammar3} holds.
\end{proof}

From Lemma~\ref{grammar} one obtains the generating series
\begin{align}
G_A (z) &= 1 + 2z G_A(z), \label{ga}\\
G_B (z) &= 1 + z G_A(z) + 2zG_B(z), \label{gb}\\
G_{\SetTriword} (z) &= 1 + z G_A(z) + zG_B(z)
\end{align}
of $A$, $B$, and $\SetTriword$.
We deduce that $\SetTriword$ admits 
\begin{equation}
G_{\SetTriword} (z) = \frac{(1-z)^2}{(1-2z)^2}
\end{equation}
as generating function.
Therefore, for any $n \geq 1$, the number of triwords is
\begin{equation}\label{equ:nbtriword}
\# \SetTriword(n) = 2^{n-2} (n + 3).
\end{equation}

\begin{Lemma}
For any $n \geq 1$, the image $\rho(\F(n))$ coincides with $\SetTriword(n)$.
\end{Lemma}

\begin{proof}
Let $d \in \F(n)$ such that $\rho (d):= u$. Then the first letter of $u$ is either $0$ or $1$. Besides, a letter $0$ cannot be follows by a letter $1$ because the height of the valleys in $d$ is weakly decreasing from left to right. Thus, one has $u \in \SetTriword(n)$. 
\smallbreak

Moreover, we know from~\cite{Cha20} that the number of elements in $\F(n)$ is~\eqref{equ:nbtriword}. 
\end{proof}

\subsection{Isomorphism of posets}
We endow the set of triwords with the componentwise order and show that the bijection $\rho$ is an isomorphism of posets. Then, we describe the meet and join of the poset so defined.
\smallbreak

For any $n \geq 1$, let $\Leq$ be the partial order on $\SetTriword(n)$ satisfying $u \Leq v$ for any $u, v \in \SetTriword(n)$ such that $u_i \leq v_i$ for all $i \in [n]$.
The set $\SetTriword(n)$ endowed with $\Leq$ is the \Def{Hochschild poset} of order $n$. 
\smallbreak

We set that $u \Covered v$ if and only if $u \Leq v$ and there is only one index $i$ such that $u_i < v_i$, and if there is $w \in \SetTriword(n)$ such that $u \Leq w \Leq v$, then either $w = u$ or $w = v$.
Obviously, the binary relation $\Covered$ is contained in the covering relation of $(\SetTriword(n),\Leq)$.
\smallbreak

Note that the minimal element of $\SetTriword(n)$ is $0^n$ and the maximal element is $12^{n-1}$.
\smallbreak

\begin{Proposition}\label{prop:coveredHochschild}
For any $n \geq 1$, the binary relation $\Covered$ is the covering relation of the Hochschild poset $\SetTriword(n)$.
\end{Proposition}

\begin{proof}
Let $u, v \in \SetTriword(n)$ such that $v$ covers $u$. The case $n = 1$ is clear. Let $n > 1$ and let $i$ be the minimal index such that $u_i \ne v_i$, and let $w := u_1 \dots u_{i-1} v_i u_{i+1} \dots u_{n}$ be the word with the same letters as $u$, except for the $i$-th letter. Since $v_i > u_i$, either $w$ is obtained by replacing in $u$ the $i$-th letter $0$ by $1$ or by $2$, or by replacing in $u$ the $i$-th letter $1$ by $2$. In both cases, $v_i$ is not $0$. Moreover, since $i$ is the minimal index such that $u_i \ne v_i$, if there is a letter $0$ before $u_i$ in $u$, then this letter exist also in $v$, and so $v_i$ cannot be $1$. Therefore, the subword $01$ cannot be generated in $w$. Thus, the word $w$ is a triword. It follows that there is a triword $w' \Leq w$ such that $u$ is covered by $w'$. One can conclude that between two triwords in covering relation, there is exactly one different letter.
\end{proof}

For any Dyck path $d =  p 10^{m} x s$ with $m >0$, $p$ a prefix, $s$ a suffix, and $x$ a movable subpath, let $N(d,x)$ be the number of consecutive $0$ letters that appear before $x$ in $d$.
\smallbreak

\begin{Proposition}\label{prop:isomorphism}
For any $n \geq 1$, the map $\rho$ is an isomorphism of posets from $\F(n)$ to $\SetTriword(n)$.
\end{Proposition}

\begin{proof}
Let $d,b \in \F(n)$. We know (Lemma 9.9 from~\cite{Cha20}) that if $d$ covers $b$ in $\F(n)$ then the words $\rho(b)$ and $\rho(d)$ differ by exactly one letter, which increases. This implies that $\rho(b)\Leq \rho(d)$.
\smallbreak

Let $u,v \in \SetTriword(n)$ such that $u \Covered v$, and let $b$ and $d$ be the respective images of $u$ and $v$ by $\rho^{-1}$. 
Since $u \Covered v$, there is only one index $i$ such that $u_i < v_i$. Then, there are three cases: either $0$ becomes $1$ or $0$ becomes $2$, or $1$ becomes $2$. 
\begin{itemize}
\item Suppose that $u_i = 0$ and $v_i = 1$. Then, in the path $b$, there is a movable subpath $x$ (in blue (dark) in \eqref{equ:cover1}) starting at the height $0$ such that $N(b,x) \geq 2$. The height of the starting point of $x$ gives the value of $u_i$ in $u$ by the map $\rho$. In the path $d$, since only one letter changes between $u$ and $v$, the same subpath $x$ starts at the height $1$ and $N(d,x) = N(b,x) - 1$. Because of this move, we have to add one $0$ after $x$. 
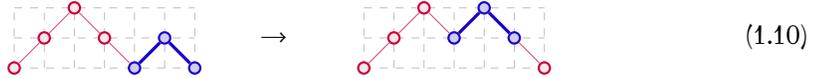
\begin{equation}\label{equ:cover1}
\scalebox{1}{
		\begin{tikzpicture}[Centering,scale=.4]
            \draw[Grid](0,0)grid(6,2);
            \node[PathNode](0)at(0,0){};
            \node[PathNode](1)at(1,1){};
            \node[PathNode](2)at(2,2){};
            \node[PathNode](3)at(3,1){};
            \node[PathNode](4)at(4,0)[circle,draw=Col4,fill=Col4!20,thick,inner sep=0pt,minimum size=1.5mm]{};
            \node[PathNode](5)at(5,1)[circle,draw=Col4,fill=Col4!20,thick,inner sep=0pt,minimum size=1.5mm]{};
            \node[PathNode](6)at(6,0)[circle,draw=Col4,fill=Col4!20,thick,inner sep=0pt,minimum size=1.5mm]{};
            \draw[PathStep](0)--(1);
            \draw[PathStep](1)--(2);
            \draw[PathStep](2)--(3);
            \draw[PathStep](3)--(4);
            \draw[PathStep](4)--(5)[color=Col4,very thick];;
            \draw[PathStep](5)--(6)[color=Col4,very thick];;
        \end{tikzpicture}}
        \qquad \rightarrow \qquad
        \scalebox{1}{
        \begin{tikzpicture}[Centering,scale=.4]
            \draw[Grid](0,0)grid(6,2);
            \node[PathNode](0)at(0,0){};
            \node[PathNode](1)at(1,1){};
            \node[PathNode](2)at(2,2){};
            \node[PathNode](3)at(3,1)[circle,draw=Col4,fill=Col4!20,thick,inner sep=0pt,minimum size=1.5mm]{};
            \node[PathNode](4)at(4,2)[circle,draw=Col4,fill=Col4!20,thick,inner sep=0pt,minimum size=1.5mm]{};
            \node[PathNode](5)at(5,1)[circle,draw=Col4,fill=Col4!20,thick,inner sep=0pt,minimum size=1.5mm]{};
            \node[PathNode](6)at(6,0){};
            \draw[PathStep](0)--(1);
            \draw[PathStep](1)--(2);
            \draw[PathStep](2)--(3);
            \draw[PathStep](3)--(4)[color=Col4,very thick];;
            \draw[PathStep](4)--(5)[color=Col4,very thick];;
            \draw[PathStep](5)--(6);
        \end{tikzpicture}}
\end{equation}

\item Suppose that $u_i = 0$ and $v_i = 2$. Then, in the path $b$, there is a movable subpath $x$ (in blue (dark) in \eqref{equ:cover2}) starting at the height $0$, followed by an other subpath $y$ also starting at the height $0$. This is the height of the starting point of $y$ which gives $u_i$ in $u$ by the map $\rho$. In the path $d$, there is a subpath $z$ starting at the height $0$ followed by the subpath $y$ which is unchanged, such that $N(d,x)=0$ and $N(d,y) = N(b,x) + N(b,y)$.
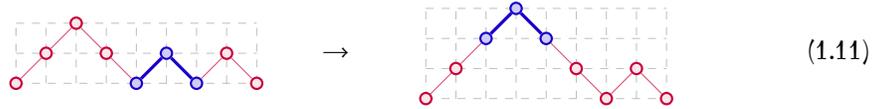
\begin{equation}\label{equ:cover2}
\scalebox{1}{
		\begin{tikzpicture}[Centering,scale=.4]
            \draw[Grid](0,0)grid(8,2);
            \node[PathNode](0)at(0,0){};
            \node[PathNode](1)at(1,1){};
            \node[PathNode](2)at(2,2){};
            \node[PathNode](3)at(3,1){};
            \node[PathNode](4)at(4,0)[circle,draw=Col4,fill=Col4!20,thick,inner sep=0pt,minimum size=1.5mm]{};
            \node[PathNode](5)at(5,1)[circle,draw=Col4,fill=Col4!20,thick,inner sep=0pt,minimum size=1.5mm]{};
            \node[PathNode](6)at(6,0)[circle,draw=Col4,fill=Col4!20,thick,inner sep=0pt,minimum size=1.5mm]{};
            \node[PathNode](7)at(7,1){};
            \node[PathNode](8)at(8,0){};
            \draw[PathStep](0)--(1);
            \draw[PathStep](1)--(2);
            \draw[PathStep](2)--(3);
            \draw[PathStep](3)--(4);
            \draw[PathStep](4)--(5)[color=Col4,very thick];
            \draw[PathStep](5)--(6)[color=Col4,very thick];
            \draw[PathStep](6)--(7);
            \draw[PathStep](7)--(8);
        \end{tikzpicture}}
        \qquad \rightarrow \qquad
        \scalebox{1}{
        \begin{tikzpicture}[Centering,scale=.4]
            \draw[Grid](0,0)grid(8,3);
            \node[PathNode](0)at(0,0){};
            \node[PathNode](1)at(1,1){};
            \node[PathNode](2)at(2,2)[circle,draw=Col4,fill=Col4!20,thick,inner sep=0pt,minimum size=1.5mm]{};
            \node[PathNode](3)at(3,3)[circle,draw=Col4,fill=Col4!20,thick,inner sep=0pt,minimum size=1.5mm]{};
            \node[PathNode](4)at(4,2)[circle,draw=Col4,fill=Col4!20,thick,inner sep=0pt,minimum size=1.5mm]{};
            \node[PathNode](5)at(5,1){};
            \node[PathNode](6)at(6,0){};
            \node[PathNode](7)at(7,1){};
            \node[PathNode](8)at(8,0){};
            \draw[PathStep](0)--(1);
            \draw[PathStep](1)--(2);
            \draw[PathStep](2)--(3)[color=Col4,very thick];
            \draw[PathStep](3)--(4)[color=Col4,very thick];
            \draw[PathStep](4)--(5);
            \draw[PathStep](5)--(6);
            \draw[PathStep](6)--(7);
            \draw[PathStep](7)--(8);
        \end{tikzpicture}}
\end{equation}

\item Suppose that $u_i = 1$ and $v_i = 2$. This case is very similar to the previous case, by changing the height of the starting point $0$ of $x$, $y$ and $z$ by $1$.
\end{itemize}

In all cases, one has $b \Leq_{\DexterOrder} d$.
\end{proof}

Let us describe the join and the meet between two triwords $u$ and $v$.
\smallbreak

Let $u,v \in \SetTriword(n)$, and let $r := \mathrm{max}(u_1, v_1) \dots\mathrm{max}(u_n, v_n)$. Since $u_1$ and $v_1$ are both none $2$, $r_1 \ne 2$. Besides, if $r_i = 0$ for $i\in [n]$, then necessarily $u_i$ and $v_i$ have to be equal to $0$. In this case, for all $j>i$, neither $u_j$ nor $v_j$ can take the value $1$. Therefore, if there is an index $i\in[n]$ such that $r_i = 0$, then $r_j \ne 1$ for all $j>i$. Thus $r$ is a triword.
\smallbreak

The triword $r$ is the join between $u$ and $v$. Indeed, $r$ is by definition the smallest element such that for all $i\in[n]$, $r_i \geq u_i$ and $r_i \geq v_i$. Moreover, since the join between $u$ and $v$ is unique, by Proposition~\ref{prop:isomorphism}, the Hochschild poset is a join-semilattice. One can conclude that Hochschild poset is a lattice since there is a unique minimal triword~\cite{Sta11}\footnote{This fact is already known since the Hochschild poset is an interval of the dexter meet-semilattice~\cite{Cha20}.}.
\smallbreak

Let $s := \mathrm{min}(u_1, v_1) \dots \mathrm{min}(u_n, v_n)$.
The word $s$ is not necessarily a triword. For instance, if we consider $u = 11112$ and $v = 10022$, two triwords of size $5$, then $s = 10012$ which contains a subword $01$. 
\smallbreak

Let $t := u \Meet v$ be the word obtained from $s$ by changing all subwords $01$ by $00$ in $s$.
\smallbreak

\begin{Proposition}\label{prop:meetdefinition-skyscrapers}
Let $n \geq 1$ and $u,v \in \SetTriword(n)$, then $t := u \Meet v$ is the meet between $u$ and~$v$.
\end{Proposition}

\begin{proof}
If $s := \mathrm{min}(u_1, v_1) \dots \mathrm{min}(u_n, v_n)$ is a triword, then $t = s$. Suppose that $s$ is not a triword.
Since we replace in $s$ all subwords $01$ by $00$, $t$ is a triword. Moreover, if there is a subword $01$ in $s$, then either $u$ or $v$ has a letter $0$ following by letters $0$ or $2$. Necessary, the word $s$ inherits this letter $0$, and then $t$ is a triword if all letters after this letter $0$ are $0$ or $2$. 
Therefore, the triword $t$ is the greatest element such that
$t \leq u$ and $t \leq v$.
\end{proof}
\smallbreak

For example, in order to compute $11112 \Meet 10222$, first we compute $s = 10112$, which is not a triword.  We replace the subword $s_2s_3$ and $s_2s_4$ by the subword $00$. One has $11112 \Meet 10222 = 10002$. 
\smallbreak

\section{Geometric properties}\label{sec:geom-prop}
Through triwords, it is possible to give a cubic realization of the Hochschild lattice by placing in the space $\mathbb{R}^n$ all triwords of size $n$. This lattice thus joins the family of posets having a cubic realisation~\cite{Com19, CG20}. 
This realization allows us to show two geometrical results: on the one hand that the Hochschild lattice is EL-shellable and on the other hand that this lattice is constructible by interval doubling.

\subsection{Cubic realizations}
The Hochschild poset $\SetTriword(n)$ can be seen as a geometric object in the space $\R^n$ by placing for each $u \in
\SetTriword(n)$ a vertex of coordinates $\Par{u_1,\dots, u_n}$, and by forming for each $u, v \in \SetTriword(n)$ such that $u
\Covered v$ an edge between $u$ and $v$. We call \Def{cubic realization} of $\SetTriword(n)$ the geometric object $\CubicReal\Par{\SetTriword(n)}$ just defined. Figure~\ref{fig:examples_triword_posets} shows the cubic realization of the poset $\SetTriword(2)$ and the poset $\SetTriword(3)$.
\smallbreak

\begin{figure}[ht]
    \centering
    \subfloat[][$\CubicReal\Par{\SetTriword(2)}$.]{
    \centering
    \scalebox{1}{
    \begin{tikzpicture}[Centering,xscale=1.1,yscale=1.1,rotate=-135]
        \draw[Grid](0,0)grid(1,2);
        %
        %
        \node[NodeGraph](00)[]at(0,0){};
        \node[NodeGraph](10)[]at(1,0){};
        \node[NodeGraph](11)[]at(1,1){};
        \node[NodeGraph](02)[]at(0,2){};
        \node[NodeGraph](12)[]at(1,2){};
        \node[NodeLabelGraph,above of=00]{$00$};
        \node[NodeLabelGraph,left of=10]{$10$};
        \node[NodeLabelGraph,left of=11]{$11$};
        \node[NodeLabelGraph,right of=02]{$02$};
        \node[NodeLabelGraph,below of=12]{$12$};
        \draw[EdgeGraph](00)--(10);
        \draw[EdgeGraph](00)--(02);
        \draw[EdgeGraph](10)--(11);
        \draw[EdgeGraph](11)--(12);
        \draw[EdgeGraph](02)--(12);
    \end{tikzpicture}}
    \label{subfig:triword_poset_2}}
    \hspace{2cm}
    \subfloat[][$\CubicReal\Par{\SetTriword(3)}$.]{
    \centering
    \scalebox{1}{
    \begin{tikzpicture}[Centering,xscale=1.3,yscale=1.3,
        y={(0,-.5cm)}, x={(-1.0cm,-1.0cm)}, z={(1.0cm,-1.0cm)}]
        \DrawGridSpace{1}{2}{2}
        %
        %
        \node[NodeGraph](000)[]at(0,0,0){};
        \node[NodeGraph](020)[]at(0,2,0){};
        \node[NodeGraph](100)[]at(1,0,0){};
        \node[NodeGraph](002)[]at(0,0,2){};
        \node[NodeGraph](022)[]at(0,2,2){};
        \node[NodeGraph](110)[]at(1,1,0){};
        \node[NodeGraph](111)[]at(1,1,1){};
        \node[NodeGraph](020)[]at(0,2,0){};
        \node[NodeGraph](120)[]at(1,2,0){};
        \node[NodeGraph](121)[]at(1,2,1){};
        \node[NodeGraph](122)[]at(1,2,2){};
        \node[NodeGraph](112)[]at(1,1,2){};
        \node[NodeGraph](102)[]at(1,0,2){};
        \node[NodeLabelGraph,above of=000]{$000$};
        \node[NodeLabelGraph,left of=100]{$100$};
        \node[NodeLabelGraph,right of=002]{$002$};
        \node[NodeLabelGraph,right of=022]{$022$};
        \node[NodeLabelGraph,left of=110]{$110$};
        \node[NodeLabelGraph,left of=111]{$111$};
        \node[NodeLabelGraph,below of=020]{$020$};
        \node[NodeLabelGraph,left of=120]{$120$};
        \node[NodeLabelGraph,left of=121]{$121$};
        \node[NodeLabelGraph,below of=122]{$122$};
        \node[NodeLabelGraph,right of=112]{$112$};
        \node[NodeLabelGraph,above of=102]{$102$};
        \draw[EdgeGraph](000)--(020);
        \draw[EdgeGraph](000)--(100);
        \draw[EdgeGraph](000)--(002);
        \draw[EdgeGraph](100)--(110);
        \draw[EdgeGraph](100)--(102);
        \draw[EdgeGraph](002)--(102);
        \draw[EdgeGraph](002)--(022);
        \draw[EdgeGraph](020)--(120);
        \draw[EdgeGraph](020)--(022);
        \draw[EdgeGraph](110)--(120);
        \draw[EdgeGraph](110)--(111);
        \draw[EdgeGraph](111)--(121);
        \draw[EdgeGraph](111)--(112);
        \draw[EdgeGraph](120)--(121);
        \draw[EdgeGraph](121)--(122);
        \draw[EdgeGraph](102)--(112);
        \draw[EdgeGraph](112)--(122);
        \draw[EdgeGraph](022)--(122);
    \end{tikzpicture}}
    \label{subfig:triword_poset_3}}
  \caption{\footnotesize Cubic realizations of some Hochschild posets.}
    \label{fig:examples_triword_posets}
\end{figure}
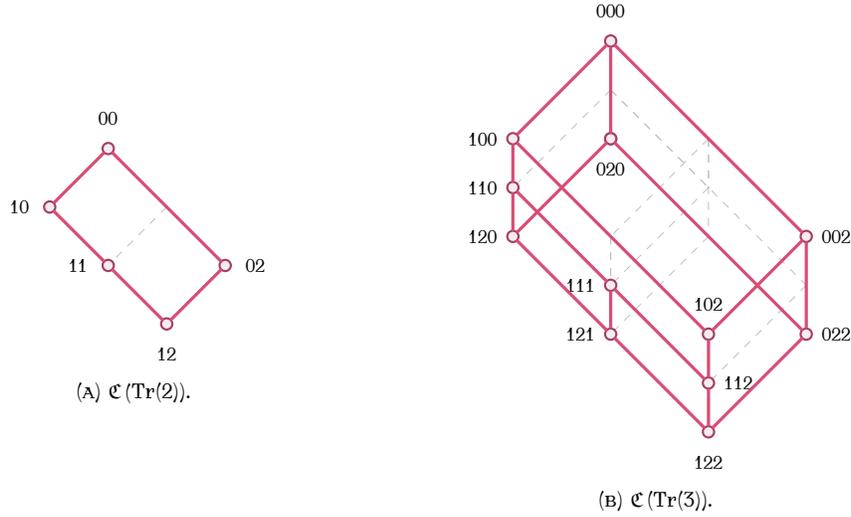

The first thought that comes to mind, is that for any $n \geq 1$, any $k$-face of the realization $\CubicReal\Par{\SetTriword(n)}$ is contained in a $n-1$-face of the hypercube of dimension $n$, for $k \in [0, n-1]$. 
Indeed, between the minimal triword $0^n := u$ and the maximal triword $12^{n-1} := v$, there is no triword $w$ of size $n$ such that $u_i < w_i < v_i$ for all $i \in [n]$ since $u_1 = 0$ and $v_1 = 1$. 
\smallbreak

Therefore, we can see this realization as one empty cell of dimension $n$. Thus, it is clear that the volume of $\CubicReal\Par{\SetTriword(n)}$ is $2^{n-1}$.
\smallbreak

\subsection{EL-shellability}
In~\cite{BW96} and~\cite{BW97}, Björner and Wachs
generalized the method of labellings of
the cover relations of graded posets to the case of non-graded posets. In particular, they
showed the EL-shellability of the Tamari poset~\cite{BW97}. In this section, we show that the Hochschild lattice is EL-shellable.
\smallbreak 

A poset $\PosetP$ is \Def{bounded} if it has a unique maximal element and a unique minimal element for $\Leq_\PosetP$. 
A \Def{chain} in $\PosetP$ is a sequence of elements
\begin{equation}\label{equ:chainEL}
x_1 \Leq_\PosetP x_2 \Leq_\PosetP \dots \Leq_\PosetP x_{r-1} \Leq_\PosetP x_r.
\end{equation}
Let $\Covered_\PosetP$ be the covering relation of $\PosetP$. If $x_i \Covered_\PosetP x_{i+1}$ for all $i \in [r-1]$, then the chain~\eqref{equ:chainEL} is \Def{saturated}.
\smallbreak

By a slight abuse of notation, the set of elements $(x,y)$ such that $x \Covered_\PosetP y$ is also denoted by $\Covered_\PosetP$.
Let $\PosetP$ be a bounded poset and $\Lambda$ be a poset, and $\lambda : \Covered_{\PosetP} \to \Lambda$ be a map. For any saturated chain $\Par{x^{(1)}, \dots, x^{(k)}}$ of
$\PosetP$, we set
\begin{equation}
    \lambda\Par{x^{(1)}, \dots, x^{(k)}}
    :=
    \Par{\lambda\Par{x^{(1)}, x^{(2)}}, \dots,
    \lambda\Par{x^{(k - 1)}, x^{(k)}}}.
\end{equation}
We say that a saturated chain of $\PosetP$ is \Def{$\lambda$-increasing}
(resp. \Def{$\lambda$-weakly decreasing}) if its image by $\lambda$ is an
increasing (resp. weakly decreasing) word for the order relation
$\Leq_{\Lambda}$. We say also that a saturated chain $\Par{x^{(1)},
\dots, x^{(k)}}$ of $\PosetP$ is \Def{$\lambda$-smaller} than a
saturated chain $\Par{y^{(1)}, \dots, y^{(k)}}$ of $\PosetP$ if $\lambda\Par{x^{(1)}, \dots, x^{(k)}}$ is smaller than $\lambda\Par{y^{(1)}, \dots, y^{(k)}}$ for the
lexicographic order induced by $\Leq_{\Lambda}$. The map $\lambda$ is called \Def{EL-labeling} (edge lexicographic labeling) of $\PosetP$ if for any $x, y \in \PosetP$ satisfying $x \Leq_\PosetP y$, there is exactly one $\lambda$-increasing saturated chain from $x$ to~$y$, and this chain is $\lambda$-minimal among all saturated chains from $x$ to $y$. Any bounded poset that admits an EL-labeling is \Def{EL-shellable} (see~\cite{BW96,BW97}).
\smallbreak

The EL-shellability of a poset $\PosetP$ implies several topological and
order theoretical properties of the associated order complex
$\Delta(\PosetP)$ built from $\PosetP$. Recall that the faces of this
simplicial complex are all the chains of $\PosetP$. For instance, if $\PosetP$ has at most one
$\lambda$-weakly decreasing chain between any pair of elements then the Möbius function of $\PosetP$ takes values in $\{-1, 0, 1\}$. In this case, the simplicial complex associated with each open interval of $\PosetP$ is either contractible or has the homotopy type of
a sphere~\cite{BW97}.
\smallbreak

In order to show the EL-shellability of $\SetTriword(n)$ for $n \geq 1$, we set $\Lambda$ as the poset $\Z^2$ ordered lexicographically. Then we introduce the map
\begin{math}
    \lambda :
    \Covered \to \Z^2
\end{math}
defined for any $u,v$ such that $u \Covered v$ by
\begin{equation} \label{equ:el_labelling}
    \lambda (u, v) := \Par{i, u_i}
\end{equation}
where $i$ is the unique index such that $u_i \ne v_i$.  Observe that because of the covering relation $\Covered$ defined in Proposition~\ref{prop:coveredHochschild}, the image by $\lambda$ of any
saturated chain in $\SetTriword(n)$ is well-defined.
\smallbreak

For any $u, v \in \SetTriword(n)$, let
\begin{equation}
    \DiffIndexes(u, v) := \Bra{d ~:~ u_d \ne v_d}
\end{equation}
be the set of all indices of different letters between $u$ and $v$.
\smallbreak

\begin{Theorem}\label{shellability}
For any $n \geq 1$, the map $\lambda$ is an EL-labelling of the Hochschild lattice $\SetTriword(n)$. Moreover, there is at most one $\lambda$-weakly decreasing chain between any pair of comparable elements of $\SetTriword(n)$.
\end{Theorem}

\begin{proof}
Let $u, v \in\SetTriword(n)$ such that $u \Leq v$ and
\begin{equation}
D(u, v) := \{ d_{1},d_{2},\dots,d_{s}\},
\end{equation}
with $d_{1}< d_2 < \dots < d_{s}$. For $k \in [s]$, let $u^{(k)}$ be the word of size $n$ defined by replacing the $k$ letters $u_{d_1}, u_{d_2}, \dots , u_{d_k}$ in $u$ by the $k$ letters $v_{d_1}, v_{d_2}, \dots , v_{d_k}$ of $v$.

Thus, for any $k \in [s]$, either $u_i^{(k)} = u_i$ or $u_i^{(k)} = v_i$ for all $i \in [n]$. Since the letters are increased from the triword $u$ from left to right, the word $u^{(k)}$ is not a triword if and only if there is a letter $u_i^{(k)} = 0$ and a letter $u_j^{(k)} = 1$ with $i \leq d_k$ and $j > i$. However, if there is a letter $u_i^{(k)} = 0$ in $u^{(k)}$ with $i \leq d_k$, then $v_i = 0$ since $u_i^{(k)} = v_i$ by construction of $u^{(k)}$. And so $u_i = 0$ since by hypothesis $u_i \leq v_i$. Thus, $u_i = 0$ and $v_i = 0$ imply respectively that $u_j \ne 1$ and $v_j \ne 1$ in the triwords $u$ and $v$ for all $j > i$. In particular, one has $u^{(k)}_j \neq 1$ for all $j > i$. 
It follows that the subword $01$ cannot occur in $u^{(k)}$, and then $u^{(k)}$ is a triword.
Let us consider the chain 
\begin{equation}
\Par{u \Leq u^{(1)} \Leq u^{(2)} \Leq \dots \Leq u^{(s-1)} \Leq u^{(s)} = v}
\end{equation}
which is not necessarily saturated. Then, by concatenating the unique saturated chain in each interval $[u^{(k-1)}, u^{(k)}]$ for all $k \in [s]$, we obtain a saturated chain between $u$ and $v$. Since each word $u^{(k)}$ of this saturated chain is obtained from $u$ by replacing letters from left to right, this chain is clearly weakly increasing for the partial order $\Leq$. 
Furthermore, between two consecutive triwords $u^{(k-1)}$ and $u^{(k)}$ in this saturated chain, $u^{(k-1)} \Covered u^{(k)}$. Therefore, the image of the chain by $\lambda$ is increasing for $\Leq$. Thus this chain is $\lambda$-increasing. 
\smallbreak

Moreover, since between any two consecutive triwords of this chain only one letter is different, if we consider another saturated chain from $u$ to $v$, then at some point, this chain passes through a word obtained by increasing a letter which has not the smallest possible index. It lead us to choose later in this chain the letter with a smallest index to increase it. For this reason, the saturated chain obtained is not $\lambda$-increasing.
\smallbreak

If a $\lambda$-weakly decreasing chain exists in $[u,v]$, then it must have the sequence of edge-labels 
\begin{equation}
\Par{(d_s, u_{d_s}), (d_{s-1}, u_{d_{s-1}}), ~\dots~, (d_2, u_{d_2}), (d_1, u_{d_1})}.
\end{equation}
Indeed, suppose that between $u$ and $v$, there is an index $d \in \DiffIndexes(u, v)$ such that $u_d = 0$ and $v_d = 2$, and there is a triword $w$ such that $u \Leq w \Leq v$ with $w_d = 1$. Then, for this index $d$, the sequence of edge-labels passing through $w$ is $\Par{(d, 0), (d, 1)}$, and so the saturated chain passing through $w$ in $[u,v]$ cannot be $\lambda$-weakly decreasing. Therefore, to obtain a $\lambda$-weakly decreasing chain in $[u,v]$, each index $d$ of $\DiffIndexes(u, v)$ can only appear once in the sequence of edge-labels.
\smallbreak

Assume that there is a $\lambda$-weakly decreasing chain. For the same reason as previously, this chain is unique.
\end{proof}

For instance, for $\SetTriword(3)$, the $\lambda$-increasing chain between $000$ and $122$ is 
\begin{equation}
\Par{000, 100, 110, 120, 121, 122},
\end{equation}
and
\begin{equation}
\lambda\Par{000, \dots, 122}
    =
    \Par{(1, 0), (2, 0), (2, 1), (3, 0), (3, 1)}.
\end{equation}
\smallbreak

For the same interval, the $\lambda$-weakly decreasing chain is
\begin{equation}
\Par{000, 002, 022, 122},
\end{equation}
and 
\begin{equation}
\lambda\Par{000, \dots, 122}
    =
    \Par{(3, 0), (2, 0), (1, 0)}.
\end{equation}

\subsection{Construction by interval doubling}

Let $\IntervalTwo$ be the poset $\{0, 1\}$ where $0 \Leq 1$.  Let
$\PosetP$ be a poset and $I$ one of its intervals. The
\Def{interval doubling} of $I$ in $\PosetP$ is the poset
\begin{equation}
    \PosetP[I] := (\PosetP \setminus I) \cup (I \times \IntervalTwo),
\end{equation}
having $\Leq'_{\PosetP}$ as order relation, which is defined as follows. For any
$x, y \in \PosetP[I]$, one has $x \Leq'_{\PosetP} y$ if one of the following
assertions is satisfied:
\begin{enumerate}[label=(\roman*)]
    \item $x \in \PosetP \setminus I$, $y \in \PosetP \setminus I$, and
    $x \Leq_{\PosetP} y$,
    \item $x \in \PosetP \setminus I$, $y = \Par{y', b} \in I \times
    \IntervalTwo$, and $x \Leq_{\PosetP} y'$,
    \item $x = \Par{x', a} \in I \times \IntervalTwo$, $y \in \PosetP
    \setminus I$, and $x' \Leq_{\PosetP} y$,
    \item $x = \Par{x', a} \in I \times \IntervalTwo$, $y = \Par{y', b}
    \in I \times \IntervalTwo$, and $x' \Leq_{\PosetP} y'$ and $a \Leq_{\PosetP} b$.
\end{enumerate}
This operation has been introduced by Alan Day as an operation on posets preserving the property of being lattices. A lattice $\LatticeL$
is \Def{constructible by interval doubling} (bounded in the original article) if $\LatticeL$ is isomorphic as a poset to a poset
obtained by performing a sequence of interval doubling from the singleton lattice. 
\smallbreak

For all $n \geq 1$, let us build $\SetTriword(n+1)$ from $\SetTriword(n)$ by following these three steps.
\begin{enumerate}[label=(\roman*)]
\item\label{step1doubling} Let $T_0(n+1)$ be the poset on the set of all words $u0$ such that $u \in \SetTriword(n)$.
\item \label{step2doubling} We build the set $T_2(n+1)$ from $T_0(n+1)$ by changing for all $u \in T_0(n+1)$ the letter $u_{n+1}$ to $2$.
Let $T_{0,2}(n+1)$ be the union $T_0(n+1) \cup T_2(n+1)$. 
\item \label{step3doubling} Let $I_0$ be the set of words of shape $1(1+2)^*0$. We build the set $I_1$ from $I_0$ by changing for all $u \in I_0$ the letter $0$ to $1$. Let $T(n+1)$ be the union $T_{0,2}(n+1) \cup I_1$.
\end{enumerate}

\begin{Lemma}\label{lem:doubling}
For any $n \geq 1$, the Hochschild poset $\SetTriword(n+1)$ is the poset $(T(n+1), \Leq)$ built from $\SetTriword(n)$.
\end{Lemma}

\begin{proof}
Let $u \in T(n+1)$, $u$ is written either $v0$, or $v2$ with $v \in \SetTriword$, or is a word of form $1(1+2)^*1$. It is clear that, for any $v \in \SetTriword(n)$, adding a letter $0$ or a letter $2$ at the end of $v$ give a triword of size $n+1$. Likewise, a word of form $1(1+2)^*1$ is also a triword.
\smallbreak

Now, let $u \in \SetTriword(n+1)$. Suppose that $u_{n+1} = 1$. Since the subword $01$ is forbidden, one has $u_i \in \{1,2\}$ for all $i \in [n]$. Therefore, $u$ belongs to $T(n+1)$. Suppose that $u_{n+1} = 0$ or that $u_{n+1} = 2$. Since $u$ belongs to $\SetTriword(n+1)$, the conditions of triwords remain on the prefix $v$ of size $n$ of $u$. Thus, one has $v \in \SetTriword(n)$.
\end{proof}

\begin{Theorem} \label{thm:bounded_lattice_triword}
For any $n \geq 1$, the Hochschild poset $\SetTriword(n)$ is constructible by interval doubling.
\end{Theorem}

\begin{proof}
We proceed by induction on $n \geq 1$. 
If $n = 1$, we have the poset $\IntervalTwo$, namely the poset with two elements, which is a lattice constructible by interval doubling. 
Assume now that $n \geq 2$. We have to show that $\SetTriword(n+1)$ can be obtained from $\SetTriword(n)$ by a sequence of interval doublings.
By Lemma~\ref{lem:doubling}, one has that $\SetTriword(n+1)$ is the poset $T(n+1)$. Since $T(n+1)$ is obtain from $\SetTriword(n)$ by performing the three steps~\ref{step1doubling},~\ref{step2doubling}, and~\ref{step3doubling}, by showing that these two last steps are two operations of interval doubling, the intended result will follow.
\smallbreak

Let us consider $T_0(n+1)$. By changing for all $u \in T_0(n+1)$ the last letter $0$ to $2$, a copy $T_2(n+1)$ of $T_0(n+1)$ is obtained. Since any $u \in T_0(n+1)$ have a copy $v \in T_2(n+1)$ such that $u_i = v_i$ for all $i \in [n]$ and $u_{n+1} \leq v_{n+1}$, one has that $u \Leq v$. Therefore, the step~\ref{step2doubling} is the doubling of the interval $T_0(n+1)$.
\smallbreak

In the step~\ref{step3doubling} one builds $I_1$ from $I_0$ by changing for all $u \in I_0$ the letter $0$ to $1$. Since for all $u, v \in I_0$ such that $u \Leq v$, any word $w$ such that $u \Leq w \Leq v$ is by definition of $\Leq$ a word of shape $1(1+2)^*0$, one has that $I_0$ is the interval $[1^n0, 12^{n-1}0]$. For the same reason, $I_1$ is the interval $[1^{n+1}, 12^{n-1}1]$. 
\smallbreak

Since any $u \in I_0$ has a copy $v \in I_1$ such that $u_i = v_i$ for all $i \in [n]$ and $u_{n+1} \leq v_{n+1}$, one has that $u \Leq v$. Meanwhile, any $u \in I_0$ has a copy $w \in T_2(n+1)$, included in the interval $[1^{n}2, 12^{n}]$, such that $u_i = w_i$ for all $i \in [n]$ and $u_{n+1} \leq w_{n+1}$. However, by construction, one has $u_{n+1} = 0$, $v_{n+1} = 1$, and $w_{n+1} = 2$, for all $u \in I_0$, $v \in I_1$ and $w \in [1^{n}2, 12^{n}]$. It follows that $u \Leq v \Leq w$ for all $u \in I_0$, $v \in I_1$ and $w \in [1^{n}2, 12^{n}]$ such that $u_i = v_i = w_i$ for $i \in [n]$. Therefore, the step~\ref{step3doubling} is the doubling of the interval $I_0$.
\end{proof}

Note that for $n = 0$, $\SetTriword(0) = \{\epsilon\}$ is constructible by interval doubling. Note also that, for any $n \geq 1$, only two steps are necessary to built $\SetTriword(n+1)$ from $\SetTriword(n)$, by starting with the doubling of $T_0(n+1)$ built from $\SetTriword(n)$,
\begin{equation}
\SetTriword(n) \simeq T_0(n+1) \rightarrow T_0(n+1) \times 2 \rightarrow \SetTriword(n+1).
\end{equation}

\begin{figure}[ht]
    \centering
    \begin{multline*}
    	\scalebox{.75}{
        \begin{tikzpicture}[Centering,xscale=0.8,yscale=0.8,rotate=-135]
        \draw[Grid](0,0)grid(1,2);
    	\node[NodeGraph](00)[]at(0,0){};
        \node[NodeGraph](10)[]at(1,0){};
        \node[NodeGraph](11)[]at(1,1){};
        \node[NodeGraph](02)[]at(0,2){};
        \node[NodeGraph](12)[]at(1,2){};
        \node[NodeLabelGraph,above of=00]{$00$};
        \node[NodeLabelGraph,left of=10]{$10$};
        \node[NodeLabelGraph,left of=11]{$11$};
        \node[NodeLabelGraph,right of=02]{$02$};
        \node[NodeLabelGraph,below of=12]{$12$};
        \draw[EdgeGraph](00)--(10);
        \draw[EdgeGraph](00)--(02);
        \draw[EdgeGraph](10)--(11);
        \draw[EdgeGraph](11)--(12);
        \draw[EdgeGraph](02)--(12);
    \end{tikzpicture}}
        \; \simeq \;
        \scalebox{.75}{
        \begin{tikzpicture}[Centering,xscale=1.1,yscale=1.1,
        y={(0,-.5cm)}, x={(-1.0cm,-1.0cm)}, z={(1.0cm,-1.0cm)}]
        \DrawGridSpace{1}{2}{2}
        %
        %
        \node[NodeGraph](000)[]at(0,0,0){};
        \node[NodeGraph](020)[]at(0,2,0){};
        \node[NodeGraph](120)[]at(1,2,0){};
        \node[NodeGraph](110)[]at(1,1,0){};
        \node[NodeGraph](100)[]at(1,0,0){};
        \node[NodeLabelGraph,above of=000]{$000$};
        \node[NodeLabelGraph,below of=020]{$020$};
        \node[NodeLabelGraph,left of=120]{$120$};
        \node[NodeLabelGraph,left of=110]{$110$};
        \node[NodeLabelGraph,left of=100]{$100$};
        \draw[EdgeGraph](000)--(100);
        \draw[EdgeGraph](000)--(020);
        \draw[EdgeGraph](100)--(110);
        \draw[EdgeGraph](110)--(120);
        \draw[EdgeGraph](020)--(120);
        \end{tikzpicture}}
         \; \to \;
        \scalebox{.75}{
        \begin{tikzpicture}[Centering,xscale=1.1,yscale=1.1,
        y={(0,-.5cm)}, x={(-1.0cm,-1.0cm)}, z={(1.0cm,-1.0cm)}]
        \DrawGridSpace{1}{2}{2}
        %
        %
        \node[NodeGraph](000)[]at(0,0,0){};
        \node[NodeGraph](020)[]at(0,2,0){};
        \node[NodeGraph](100)[]at(1,0,0){};
        \node[NodeGraph](002)[]at(0,0,2){};
        \node[NodeGraph](022)[]at(0,2,2){};
        \node[NodeGraph](110)[]at(1,1,0){};
        \node[NodeGraph](120)[]at(1,2,0){};
        \node[NodeGraph](122)[]at(1,2,2){};
        \node[NodeGraph](112)[]at(1,1,2){};
        \node[NodeGraph](102)[]at(1,0,2){};
        \node[NodeLabelGraph,above of=000]{$000$};
        \node[NodeLabelGraph,left of=100]{$100$};
        \node[NodeLabelGraph,right of=002]{$002$};
        \node[NodeLabelGraph,right of=022]{$022$};
        \node[NodeLabelGraph,left of=110]{$110$};
        \node[NodeLabelGraph,below of=020]{$020$};
        \node[NodeLabelGraph,left of=120]{$120$};
        \node[NodeLabelGraph,below of=122]{$122$};
        \node[NodeLabelGraph,above right of=112]{$112$};
        \node[NodeLabelGraph,above of=102]{$102$};
        \draw[EdgeGraph](000)--(020);
        \draw[EdgeGraph](000)--(100);
        \draw[EdgeGraph](000)--(002);
        \draw[EdgeGraph](100)--(110);
        \draw[EdgeGraph](100)--(102);
        \draw[EdgeGraph](002)--(102);
        \draw[EdgeGraph](002)--(022);
        \draw[EdgeGraph](020)--(120);
        \draw[EdgeGraph](020)--(022);
        \draw[EdgeGraph](110)--(120);
        \draw[EdgeGraph](110)--(112);
        \draw[EdgeGraph](120)--(122);
        \draw[EdgeGraph](102)--(112);
        \draw[EdgeGraph](112)--(122);
        \draw[EdgeGraph](022)--(122);
        \end{tikzpicture}}
         \; \to \;
        \scalebox{0.75}{
        \begin{tikzpicture}[Centering,xscale=1.1,yscale=1.1,
        y={(0,-.5cm)}, x={(-1.0cm,-1.0cm)}, z={(1.0cm,-1.0cm)}]
        \DrawGridSpace{1}{2}{2}
        %
        %
        \node[NodeGraph](000)[]at(0,0,0){};
        \node[NodeGraph](020)[]at(0,2,0){};
        \node[NodeGraph](100)[]at(1,0,0){};
        \node[NodeGraph](002)[]at(0,0,2){};
        \node[NodeGraph](022)[]at(0,2,2){};
        \node[NodeGraph](110)[]at(1,1,0){};
        \node[NodeGraph](111)[]at(1,1,1){};
        \node[NodeGraph](020)[]at(0,2,0){};
        \node[NodeGraph](120)[]at(1,2,0){};
        \node[NodeGraph](121)[]at(1,2,1){};
        \node[NodeGraph](122)[]at(1,2,2){};
        \node[NodeGraph](112)[]at(1,1,2){};
        \node[NodeGraph](102)[]at(1,0,2){};
        \node[NodeLabelGraph,above of=000]{$000$};
        \node[NodeLabelGraph,left of=100]{$100$};
        \node[NodeLabelGraph,right of=002]{$002$};
        \node[NodeLabelGraph,right of=022]{$022$};
        \node[NodeLabelGraph,left of=110]{$110$};
        \node[NodeLabelGraph,left of=111]{$111$};
        \node[NodeLabelGraph,below of=020]{$020$};
        \node[NodeLabelGraph,left of=120]{$120$};
        \node[NodeLabelGraph,left of=121]{$121$};
        \node[NodeLabelGraph,below of=122]{$122$};
        \node[NodeLabelGraph,above right of=112]{$112$};
        \node[NodeLabelGraph,above of=102]{$102$};
        \draw[EdgeGraph](000)--(020);
        \draw[EdgeGraph](000)--(100);
        \draw[EdgeGraph](000)--(002);
        \draw[EdgeGraph](100)--(110);
        \draw[EdgeGraph](100)--(102);
        \draw[EdgeGraph](002)--(102);
        \draw[EdgeGraph](002)--(022);
        \draw[EdgeGraph](020)--(120);
        \draw[EdgeGraph](020)--(022);
        \draw[EdgeGraph](110)--(120);
        \draw[EdgeGraph](110)--(111);
        \draw[EdgeGraph](111)--(121);
        \draw[EdgeGraph](111)--(112);
        \draw[EdgeGraph](120)--(121);
        \draw[EdgeGraph](121)--(122);
        \draw[EdgeGraph](102)--(112);
        \draw[EdgeGraph](112)--(122);
        \draw[EdgeGraph](022)--(122);
        \end{tikzpicture}}
         \end{multline*}
    \caption{\footnotesize A sequence of interval doublings from $\SetTriword(2)$ to $\SetTriword(3)$.}
    \label{fig:doubling_hoch}
\end{figure}

For instance, Figure~\ref{fig:doubling_hoch} depicts the sequence of interval doublings from $\SetTriword(2)$ to $\SetTriword(3)$. To obtain $\SetTriword(3)$ from $T_0(3)$, we have first to double the interval $T_0(3)$, then we have to double the interval $[110, 120]$.
\smallbreak

\section{Combinatorial properties}\label{sec:combi-prop}
In this section, several combinatorial and enumerative properties of the Hochschild lattice are proved. We obtain results such as the enumeration of intervals, the enumeration of $k$-chains, and the description of the degree polynomial of the Hochschild lattice.

\subsection{Irreducible elements and maximal chains}
Here we give some general properties of the Hochschild lattice, such as its degree polynomial and a description of join-irreducible and meet-irreducible elements.
\smallbreak

Recall that an element $x$ of a lattice $\LatticeL$ is
\Def{join-irreducible} (resp. \Def{meet-irreducible}) if $x$ covers
(resp. is covered by) exactly one element in $\LatticeL$. We denote by $\JoinIrreducible(\LatticeL)$ (resp. $\MeetIrreducible(\LatticeL)$) the set of join-irreducible (resp. meet-irreducible) elements of $\LatticeL$. 
Moreover, let
\begin{equation}\label{eq:chainex}
x_1 \Covered_\LatticeL x_2 \Covered_\LatticeL \dots \Covered_\LatticeL x_{r-1} \Covered_\LatticeL x_r
\end{equation}
be a saturated chain of $\LatticeL$ where $\Covered_\LatticeL$ is the covering relation of $\LatticeL$.
The~\Def{length} of the saturated chain~\eqref{eq:chainex} is $r-1$\footnote{In Section~\ref{subsec:interval}, we deal with $k$-chains, where $k$ refers not to the length of the chain but to the number of elements forming that chain.}. A longest saturated chain between the minimal element and the maximal element of $\LatticeL$ is a \Def{maximal saturated chain}.
\smallbreak

Let us describe the set of join-irreducible and meet-irreducible elements of $\SetTriword(n)$ by using the regular expression notation~\cite{Sak09} recalled in Section~\ref{subsec:Hoch}.
\smallbreak

The two possibilities of having a join-irreducible triword are either to change a letter $u_i= 1$ to $0$ such that all letters on the left of $u_i$ are letters $1$ and letters on the right of $u_i$ are $0$, or to change a letter $u_i= 2$ to $0$ such that all other letters are $0$. Indeed, suppose that we change in a triword $u$ a letter $u_i = 2$ to $1$. Since $u$ should cover just one triword, all other letters in $u$ have to be $0$. However, since the first letter in $u$ is different from $2$, there is a letter $u_{i-1}$ such that $u_{i-1} \neq 0$. Thus, $u_{i-1}$ can be also decreased.
This implies that $u$ covers more than just one triword. 
Since the subword $01$ is not allowed, the set of triwords which covers a unique triword is described by
\begin{equation}\label{equ:joinirre}
\JoinIrreducible(\SetTriword(n)) = \{ u \in \SetTriword(n) ~:~ u \in 1^+ 0^* + 0^+ 2 0^* \}.
\end{equation}
Likewise, the three possibilities of having a meet-irreducible triword are either to change a letter $1$ to $2$ or to change a letter $0$ to $1$, or to change a letter $0$ to $2$. Moreover, for all cases, the other letters which are unchanged should be as large as possible.
Thus, the set of triwords covered by a unique triword is described by
\begin{equation}\label{equ:meetirre}
\MeetIrreducible(\SetTriword(n)) = \{ u \in \SetTriword(n) ~:~ u \in 12^* 12^* + 12^*02^* + 02^* \}.
\end{equation}
\smallbreak

Note that both regular expressions~\eqref{equ:joinirre} and~\eqref{equ:meetirre} have as generating function
\begin{equation}\label{equ:genefunctirretriword}
G_{\JoinIrreducible(\SetTriword)}(z) = G_{\MeetIrreducible(\SetTriword)}(z) = \frac{z + z^2}{(1 - z)^2}.
\end{equation}
From~\eqref{equ:genefunctirretriword}, one can deduce that, for $n \geq 1$,
\begin{equation}\label{equ:nbirretriword}
\# \JoinIrreducible(\SetTriword(n)) = \# \MeetIrreducible(\SetTriword(n)) = 2n - 1.
\end{equation}
\smallbreak

Recall that a lattice $\LatticeL$ is \Def{join-semidistributive} if for all $x, y, z \in \LatticeL$, $x \JJoin y = x \JJoin z$ implies $x \JJoin y = x \JJoin (y \Meet z)$.
Likewise, a lattice $\LatticeL$ is \Def{meet-semidistributive} if for all $x, y, z \in \LatticeL$, $x \Meet y = x \Meet z$ implies $x \Meet y = x \Meet (y \JJoin z)$. A lattice $\LatticeL$ is \Def{semidistributive} if $\LatticeL$ is both join-semidistributive and meet-semidistributive.  A lattice $\LatticeL$ is \Def{distributive} if $x \Meet (y \JJoin z) = (x \Meet y) \JJoin (x \Meet z)$ (or in an equivalent way $x \JJoin (y \Meet z) = (x \JJoin y) \Meet (x \JJoin z)$).
\smallbreak

In Section~\ref{sec:geom-prop}, we have shown that the Hochschild lattice is constructible by interval doubling. However, it is known from~\cite{Day79} that lattices constructible by interval doubling are in particular semidistributive.
Moreover, a finite lattice $\LatticeL$ is constructible by interval doubling if and only if it is congruence uniform~\cite{Day79}. In particular, the number of join-irreducible elements $\JoinIrreducible(\LatticeL)$ is equal to the number of doubling steps needed to build $\LatticeL$~\cite{Muh19}.
\smallbreak

Therefore, there are two consequences of Theorem~\ref{thm:bounded_lattice_triword}. The first one is that for any $n \geq 1$, the Hochschild poset $\SetTriword(n)$ is semidistributive. 
Another consequence is that the difference of numbers of join-irreducible elements between $\SetTriword(n-1)$ and $\SetTriword(n)$ is always $2$. Indeed, $\SetTriword(n)$ is constructible by interval doubling from $\SetTriword(n-1)$ with only two steps.
\smallbreak

\begin{Lemma}\label{lem:lengthmaxichain}
For any $n \geq 1$, the length of any maximal saturated chain in the Hochschild poset $\SetTriword(n)$ is $2n - 1$. Moreover, a triword belongs to a maximal saturated chain if and only if all letters following a letter $0$ are also $0$.
\end{Lemma}

\begin{proof}
If $n = 1$, then the length of the saturated chain $[0,1]$ is $1$. 
Suppose that $n >1$. 
Since all letters $0$, except the first one, can be increased to $1$, then to $2$, the length of a maximal saturated chain in $\SetTriword(n)$ between $0^n$ and $12^{n-1}$ is at most $2n - 1$.
Therefore, to obtain a maximal saturated chain between $0^n$ and $12^{n-1}$, all letters $0$ in $0^n$ must become $1$ before becoming $2$, except for the first $0$. Considering that, the letters have to be increased from left to right, in order to avoid the forbidden subword $01$. This way, each letter of $0^n$, except the first one, contributes $2$ in the length of the saturated chain between the minimal triword and the maximal triword. Since the first $0$ contributes $1$, the length of such a saturated chain is $2n - 1$.
\smallbreak

Furthermore, since the letters have to be increased from left to right, this implies that a triword $u$ belongs to a maximal saturated chain if and only if for any letter $u_i = 0$ then $u_j = 0$ for all $j \geq i$.
\end{proof}

Let $\LatticeL$ be a lattice such that the length of a maximal saturated chain is $k$. If $\# \JoinIrreducible(\LatticeL) = \# \MeetIrreducible(\LatticeL) = k$ then $\LatticeL$ is an \Def{extremal lattice}~\cite{Mar92}. 
By Lemma~\ref{lem:lengthmaxichain} and the generating function~\eqref{equ:genefunctirretriword}, one has the following result.
%

\begin{Proposition}
For any $n \geq 1$, the Hochschild lattice $\SetTriword(n)$ is extremal.
\end{Proposition}

Let us recall two definitions.
An element $x$ of a lattice $\LatticeL$ is \Def{left modular}~\cite{BS97} if for any $y \Leq_{\LatticeL} z$,
\begin{equation}
(y \JJoin x) \Meet z = y \JJoin (x \Meet z).
\end{equation}
A lattice is \Def{left modular} if there is a maximal saturated chain of left modular elements.
A lattice is \Def{trim}~\cite{Tho06} if it is an extremal left modular lattice.
\smallbreak

It is shown in~\cite{TW19} that if a lattice is extremal and semidistributive, then it is also left modular, and therefore trim. By Theorem~\ref{thm:bounded_lattice_triword}, one has that $\SetTriword(n)$ is semidistributive, thus $\SetTriword(n)$ is trim.
\smallbreak

Let $\LatticeL$ be an extremal lattice. The union of maximal saturated chains of $\LatticeL$ is known as the \Def{spine} of $\LatticeL$. It is known from~\cite{Tho06} that the spine of an extremal lattice is a distributive sublattice of $\LatticeL$. The spine of $\LatticeL$ is denoted by $\Spine(\LatticeL)$. Figure~\ref{fig:examples_spines_triword_posets} shows the spine of $\Spine(\SetTriword(2))$ and $\Spine(\SetTriword(3))$. 
\smallbreak

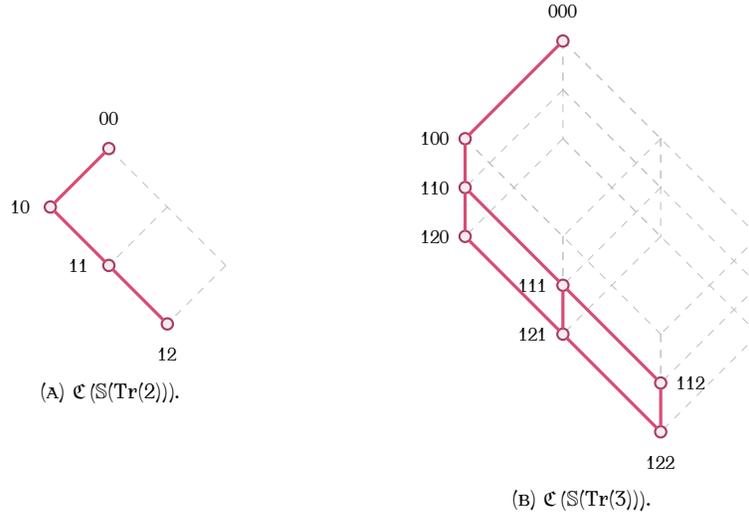
\begin{figure}[ht]
    \centering
    \subfloat[][$\CubicReal\Par{\Spine(\SetTriword(2))}$.]{
    \centering
    \scalebox{1}{
    \begin{tikzpicture}[Centering,xscale=1.1,yscale=1.1,rotate=-135]
        \draw[Grid](0,0)grid(1,2);
        %
        %
        \node[NodeGraph](00)[]at(0,0){};
        \node[NodeGraph](10)[]at(1,0){};
        \node[NodeGraph](11)[]at(1,1){};
        \node[NodeGraph](12)[]at(1,2){};
        \node[NodeLabelGraph,above of=00]{$00$};
        \node[NodeLabelGraph,left of=10]{$10$};
        \node[NodeLabelGraph,left of=11]{$11$};
        \node[NodeLabelGraph,below of=12]{$12$};
        \draw[EdgeGraph](00)--(10);
        \draw[EdgeGraph](10)--(11);
        \draw[EdgeGraph](11)--(12);
    \end{tikzpicture}}
    \label{subfig:triword_spine_poset_2}}
    \hspace{2cm}
    \subfloat[][$\CubicReal\Par{\Spine(\SetTriword(3))}$.]{
    \centering
    \scalebox{1}{
    \begin{tikzpicture}[Centering,xscale=1.3,yscale=1.3,
        y={(0,-.5cm)}, x={(-1.0cm,-1.0cm)}, z={(1.0cm,-1.0cm)}]
        \DrawGridSpace{1}{2}{2}
        %
        %
        \node[NodeGraph](000)[]at(0,0,0){};
        \node[NodeGraph](100)[]at(1,0,0){};
        \node[NodeGraph](110)[]at(1,1,0){};
        \node[NodeGraph](111)[]at(1,1,1){};
        \node[NodeGraph](120)[]at(1,2,0){};
        \node[NodeGraph](121)[]at(1,2,1){};
        \node[NodeGraph](122)[]at(1,2,2){};
        \node[NodeGraph](112)[]at(1,1,2){};
        \node[NodeLabelGraph,above of=000]{$000$};
        \node[NodeLabelGraph,left of=100]{$100$};
        \node[NodeLabelGraph,left of=110]{$110$};
        \node[NodeLabelGraph,left of=111]{$111$};
        \node[NodeLabelGraph,left of=120]{$120$};
        \node[NodeLabelGraph,left of=121]{$121$};
        \node[NodeLabelGraph,below of=122]{$122$};
        \node[NodeLabelGraph,right of=112]{$112$};
        \draw[EdgeGraph](000)--(100);
        \draw[EdgeGraph](100)--(110);
        \draw[EdgeGraph](110)--(120);
        \draw[EdgeGraph](110)--(111);
        \draw[EdgeGraph](111)--(121);
        \draw[EdgeGraph](111)--(112);
        \draw[EdgeGraph](120)--(121);
        \draw[EdgeGraph](121)--(122);
        \draw[EdgeGraph](112)--(122);
    \end{tikzpicture}}
    \label{subfig:triword_spine_poset_3}}
  \caption{\footnotesize Cubic realizations of some spines of Hochschild posets.}
    \label{fig:examples_spines_triword_posets}
\end{figure}

Let $\PosetP$ be a poset. Recall that an \Def{order ideal} in $\PosetP$ is a subset $\mathcal{S}$ of $\PosetP$ such that if $x \in \mathcal{S}$ and $y \Leq_\PosetP x$ then $y \in \mathcal{S}$.
The Fundamental theorem for finite distributive
lattices\footnote{FTFDL for short.} due to Birkhoff states that any finite distributive lattice $\LatticeL$ is isomorphic
to the lattice $\JLattice(\PosetP)$ of the order ideals of the subposet $\PosetP$ of $\LatticeL$ restricted to its join-irreducible elements, ordered by inclusion~\cite{Bir37, Sta11}.
\smallbreak

Let us consider the subposet $\JoinIrreducible(\Spine(\SetTriword(n)))$ of $\Spine(\SetTriword(n))$.
Since the spine of $\SetTriword(n)$ is a distributive sublattice of $\SetTriword(n)$, then by the FTFDL one has that $\Spine(\SetTriword(n))$ is isomorphic to $\JLattice\Par{\JoinIrreducible(\Spine(\SetTriword(n)))}$.
\smallbreak

For instance, Figure~\ref{fig:example-jj-on-spine} depicts the construction of $\JLattice\Par{\JoinIrreducible(\Spine(\SetTriword(3)))}$, which is a distributive lattice isomorphic to $\Spine(\SetTriword(3))$ (see Figure~\ref{fig:examples_spines_triword_posets}).
\smallbreak

\begin{figure}[ht]
    \centering
    \subfloat[][$\JoinIrreducible(\Spine(\SetTriword(3)))$.]{
    \centering
    \scalebox{1}{
    \begin{tikzpicture}[Centering,xscale=1.1,yscale=1.1]
        %
        \node[NodeGraph](100)[]at(-1,0){};
        \node[NodeGraph](110)[]at(0,-1){};
        \node[NodeGraph](120)[]at(-1,-2){};
        \node[NodeGraph](111)[]at(1,-2){};
        \node[NodeGraph](112)[]at(0,-3){};
        \node[NodeLabelGraph,right of=100]{$100$};
        \node[NodeLabelGraph,right of=110]{$110$};
        \node[NodeLabelGraph,right of=120]{$120$};
        \node[NodeLabelGraph,right of=111]{$111$};
        \node[NodeLabelGraph,right of=112]{$112$};
        \draw[EdgeGraph](100)--(110);
        \draw[EdgeGraph](110)--(120);
        \draw[EdgeGraph](110)--(111);
        \draw[EdgeGraph](111)--(112);
    \end{tikzpicture}}
    \label{subfig:join-irr}}
    \hspace{2cm}
    \subfloat[][$\JLattice\Par{\JoinIrreducible(\Spine(\SetTriword(3)))}$.]{
    \centering
    \scalebox{1}{
    \begin{tikzpicture}[Centering,xscale=1.1,yscale=1.1]
        %
        \node[NodeGraph](0)[]at(0,0){};
        \node[NodeGraph](1)[]at(0,-1){};
        \node[NodeGraph](2)[]at(0,-2){};
        \node[NodeGraph](3)[]at(-1,-3){};
        \node[NodeGraph](4)[]at(1,-3){};
        \node[NodeGraph](5)[]at(0,-4){};
        \node[NodeGraph](6)[]at(2,-4){};
        \node[NodeGraph](7)[]at(1,-5){};
        \node[NodeLabelGraph,left of=0]{$\emptyset$};
        \node[NodeLabelGraph,left of=1]{$\{100\}~~~~$};
        \node[NodeLabelGraph,left of=2]{$\{100, 110\}~~~~~~~~$};
        \node[NodeLabelGraph,left of=3]{$\{100, 110, 120 \}~~~~~~~~~~~~~~$};
        \node[NodeLabelGraph,right of=4]{$~~~~~~~~~~~~~~\{100, 110, 111\}$};
        \node[NodeLabelGraph,below left of=5]{$\{100, 110, 111, 120\}~~~~~~~~~~~~~~~~~~~~~~$};
        \node[NodeLabelGraph,right of=6]{$~~~~~~~~~~~~~~~~~~~~\{100, 110, 111, 112\}$};
        \node[NodeLabelGraph,right of=7]{$~~~~~~~~~~~\JoinIrreducible(\Spine(\SetTriword(3)))$};
        \draw[EdgeGraph](0)--(1);
        \draw[EdgeGraph](1)--(2);
        \draw[EdgeGraph](2)--(3);
        \draw[EdgeGraph](2)--(4);
        \draw[EdgeGraph](3)--(5);
        \draw[EdgeGraph](4)--(5);
        \draw[EdgeGraph](4)--(6);
        \draw[EdgeGraph](5)--(7);
        \draw[EdgeGraph](6)--(7);
    \end{tikzpicture}}
    \label{subfig:jj-on-spine}}
  \caption{\footnotesize Construction of $\JLattice\Par{\JoinIrreducible(\Spine(\SetTriword(3)))}$ from the poset $\JoinIrreducible(\Spine(\SetTriword(3)))$.}
    \label{fig:example-jj-on-spine}
\end{figure}

Our aim is to give a description of triwords belonging to the spine of the Hochschild lattice. Then, in this set, we give a description of join-irreducible triwords.
\smallbreak

By Lemma~\ref{lem:lengthmaxichain} we know that a triword $u$ belongs to a maximal saturated chain if and only if for any letter $u_i = 0$ then $u_j = 0$ for all $j \geq i$. Therefore, the regular expression of these triwords is
\begin{equation}\label{equ:spinetriword}
\Spine(\SetTriword(n)) = \{u \in \SetTriword(n) ~:~ u \in 0^* + 1(1+2)^*0^* \}.
\end{equation}
Therefore, the generating function is
\begin{equation}
G_{\Spine(\SetTriword)} = \frac{1}{1 - 2z},
\end{equation}
and thus
\begin{equation}
\# \Spine(\SetTriword(n)) = 2^n.
\end{equation}
\smallbreak

Let $u \in \Spine(\SetTriword(n))$. The two possibilities for $u$ to be a join-irreducible triword are either to have one unique letter $1$ which can be changed to $0$ or to have one unique letter $2$ which can be changed to $1$.
To summarize, 
\begin{equation}\label{equ:JST}
\JoinIrreducible(\Spine(\SetTriword(n))) = \{ u \in \Spine(\SetTriword(n)) ~:~ u \in 1^+0^* + 1^+20^* \}.
\end{equation}
One can deduce the generating function
\begin{equation}
G_{\JoinIrreducible(\Spine(\SetTriword))} = \frac{z + z^2}{(1-z)^2},
\end{equation}
and thus 
\begin{equation}
\# \JoinIrreducible(\Spine(\SetTriword(n))) = 2n - 1.
\end{equation}
\smallbreak

From~\eqref{equ:JST} one can also deduce that the shape of $\JoinIrreducible(\Spine(\SetTriword(n)))$ is as depicted in Figure~\ref{fig:JST(n)}.

\begin{figure}[ht]
    \centering
    \scalebox{1}{
    \begin{tikzpicture}[Centering,xscale=1.1,yscale=1.1]
        %
        \node[NodeGraph](100)[]at(-1,0){};
        \node[NodeGraph](110)[]at(0,-1){};
        \node[NodeGraph](120)[]at(-1,-2){};
        \node[NodeGraph](111)[]at(1,-2){};
        \node[NodeGraph](112)[]at(0,-3){};
        \node[NodeGraph](1111)[]at(2,-3){};
        \node[NodeGraph](1112)[]at(1,-4){};
        \node[NodeGraph](11111)[]at(3,-4){};
        \node[NodeGraph](11112)[]at(2,-5){};
        \node[NodeGraph](111111)[]at(4,-5){};
        \node[NodeGraph](111112)[]at(3,-6){};
        \node[NodeLabelGraph,right of=100]{~~$10^{n-1}$};
        \node[NodeLabelGraph,right of=110]{~~~~$1^20^{n-2}$};
        \node[NodeLabelGraph,right of=120]{~~$120^{n-2}$};
        \node[NodeLabelGraph,right of=111]{~~~~$1^30^{n-3}$};
        \node[NodeLabelGraph,right of=112]{~~~~$1^{2}20^{n-3}$};
        \node[NodeLabelGraph,right of=1111]{~~~~$1^{n-2}0^2$};
        \node[NodeLabelGraph,right of=1112]{~~~~$1^{n-3}20^2$};
        \node[NodeLabelGraph,right of=11111]{~~$1^{n-1}0$};
        \node[NodeLabelGraph,right of=11112]{~~~~$1^{n-2}20$};
        \node[NodeLabelGraph,right of=111111]{$1^n$};
        \node[NodeLabelGraph,right of=111112]{~~$1^{n-1}2$};
        \draw[EdgeGraph](100)--(110);
        \draw[EdgeGraph](110)--(120);
        \draw[EdgeGraph](110)--(111);
        \draw[EdgeGraph](111)--(112);
        \draw[EdgeGraph](111)--(1111)[dotted];
        \draw[EdgeGraph](1111)--(1112);
        \draw[EdgeGraph](1111)--(11111);
        \draw[EdgeGraph](11111)--(11112);
        \draw[EdgeGraph](11111)--(111111);
        \draw[EdgeGraph](111111)--(111112);
    \end{tikzpicture}}
  \caption{\footnotesize Shape of the poset $\JoinIrreducible(\Spine(\SetTriword(n)))$.}
    \label{fig:JST(n)}
\end{figure}

\subsection{Degree polynomial}

For any poset $\PosetP$, the \Def{degree polynomial} of $\PosetP$ is the
polynomial $\DegreePolynomial_\PosetP(x, y) \in \K \Han{x, y}$ defined
by
\begin{equation}
    \DegreePolynomial_{\PosetP}(x, y)
    :=
    \sum_{u \in \PosetP} x^{\In_\PosetP(u)} \; y^{\Out_\PosetP(u)},
\end{equation}
where for any $u \in \PosetP$, $\In_\PosetP(u)$ (resp.
$\Out_\PosetP(u)$) is the number of elements covered by (resp. covering)
$u$ in $\PosetP$. The specialization $\DegreePolynomial_\PosetP(1, y)$ is the \Def{$h$-polynomial} of~$\PosetP$.
\smallbreak

Besides, for any letter $u_i$ of $u$ with $i \in [n]$, the number of letters $u'_i$ such that the word $u'$ defined by $u'_j := u_j$ for all $j \ne i$ is in covering relation with $u$ is the \Def{degree} of the letter $u_i$. The sum of the degrees of all letters of $u$ is the number of elements covered by $u$ or covering $u$, namely $\In_\PosetP(u) + \Out_\PosetP(u)$.
\smallbreak

\begin{Proposition}\label{h-poly}
For any $n \geq 1$, the $h$-polynomial of $\SetTriword(n)$ is
\begin{equation}
\DegreePolynomial_{\SetTriword(n)}(1, y) = \Par{y + 1}^{n-2} \Par{y^2 + (n+1) y + 1}.
\end{equation}
\end{Proposition}

\begin{proof}
Let us compute the generating series
\begin{equation}
P_{\SetTriword} (y,z) := \sum_{n \geq 0} \DegreePolynomial_{\SetTriword(n)}(1, y) z^n
\end{equation}
of all degree polynomials of $\SetTriword(n)$ for all $n\geq 0$.
\smallbreak

Let us consider the grammar of $\SetTriword$ given by Lemma~\ref{grammar}.
By the map $u \mapsto z^{|u|} y^{\Out_\SetTriword(u)}$ one obtains the system of formal series
\begin{equation}
\begin{split}
P_A(y,z) &= 1 + yz P_A(y,z) + z P_A(y,z), \\
P_B(y,z) &= 1 + yz P_A(y,z) + yz P_B(y,z) + z P_B(y,z), \\
P_{\SetTriword}(y,z) &= 1 + yz P_A(y,z) + z P_B(y,z).
\end{split}
\end{equation}
Indeed, in~\eqref{grammar1} of the grammar, $0A$ becomes $yz P_A(y,z)$ because the letter $0$ can always be increased to $2$. Note that the letter $0$ in $0A$ cannot be increased to $1$ because in~\eqref{grammar3}, this expression $0A$ comes after a first letter $0$, and the subword $01$ is prohibited by definition of triwords. However, $2A$ becomes $z P_A(y,z)$ since the letter $2$ cannot be increased. Likewise, in~\eqref{grammar2}, $0A$ becomes $yz P_A (y,z)$ because the letter $0$ can be increased to $1$, and $1B$ becomes $yz P_B(y,z)$ because the letter $1$ can be increased to $2$, unlike the letter $2$ in $2B$ which becomes $z P_B(y,z)$.
\smallbreak

Thus, 
\begin{equation}
\begin{split}
P_A(y,z) &= \frac{1}{1 - z - yz}, \\
P_B(y,z) &= \frac{1-z}{(1 - z - yz)^2}, \\
P_{\SetTriword}(y,z) &= \frac{1 - z}{1 - (z + yz)} + \frac{z-z^2}{\Par{1 - (z + yz)}^2}.
\end{split}
\end{equation}
From this expression of $P_{\SetTriword}(y,z)$ in partial fraction decomposition, we deduce by a straightforward computation the given expression for $\DegreePolynomial_{\SetTriword(n)}(1, y)$.
\end{proof}

\begin{Lemma}\label{n-regulier}
For any $n \geq 1$ and $u \in \SetTriword(n)$, $\In_\SetTriword(u) + \Out_\SetTriword(u) = n$. 
\end{Lemma}

\begin{proof}
Suppose that the first letter of $u$ is $0$, then all letters of $u$ are either $0$ or $2$. The letter $u_1$ can be increased to $1$, but since we cannot have a letter $0$ followed by $1$, all other letters $0$ can only be increased to $2$, and all letters $2$ can only be decreased to $0$. And so for the case where $u_1 = 0$, all letters of $u$ have degree $1$. 
\smallbreak

Suppose now that the first letter of $u$ is $1$. Either $u_1$ is the only letter $1$ in $u$ or there is another letter $u_i = 1$ such that all letters after $u_i$ are not $1$. In the first case, $u_1$ can be decreased to $0$, thus all letters of $u$ have degree $1$. In the second case, since there is at least one other letter $1$ in $u$, $u_1$ cannot be decreased to $0$. Then the degree of $u_1$ is $0$. However, this degree is compensated by the degree of the letter $u_i$. Indeed, the last letter $1$ is the only one which can be decreased to $0$ or increased to $2$. Hence the degree of $u_i$ is $2$, and since all other letters of $u$ have degree $1$, the sum is equal to $n$. 
\end{proof}

By Proposition~\ref{h-poly} and Lemma~\ref{n-regulier}, one can deduce the degree polynomial of $(\SetTriword(n), \Leq~)$ by replacing $y^k$ in the $h$-polynomial by $x^{n-k}y^{k}$, with $k \in [0, n]$. Thus, the degree polynomial of $(\SetTriword(n), \Leq)$ is
\begin{equation}
\DegreePolynomial_{\SetTriword(n)}(x, y) = (x + y )^{n-2} \Par{x^2 + (n+1) x y + y^2}.
\end{equation}
\smallbreak


\subsection{Intervals and $k$-chains}\label{subsec:interval}
This section also provides enumerative results about the Hochschild lattice. We have already computed the length of any maximal chain for this lattice in Section~\ref{sec:geom-prop}. Here we give a method to find formulas for the number of $k$-chains of this lattice.
\smallbreak

We use for a letter $a$ and a word $u$ the notation $a \in u$ if there is a letter $u_i = a$. Conversely, $a \notin u$ if all letters $u_i$ of $u$ are different from $a$. 
Thereafter, we denote by ${\mathcal{Z}}_i(n,k)$ the set of $k$-chains of $\SetTriword(n)$ that contains exactly $i$ words $u$ such that $0 \notin u$. We denote by ${\mathfrak{z}}_i(n,k)$ the cardinality of ${\mathcal{Z}}_i(n,k)$. Note that for $n = 1$, ${\mathfrak{z}}_i(1,k) = 1$ for all $i \in [0,k]$.
\smallbreak

First, we need to define a classification for all $k$-chains of size $n$.
\smallbreak

For any $n \geq 1$ and $k \geq 1$, let $u^{(1)}, u^{(2)}, \dots, u^{(k-1)}, u^{(k)}$ be $k$ triwords of size $n$ such that $u^{(1)} \Leq u^{(2)} \Leq \dots \Leq u^{(k)}$. It is always possible to classify $k$-chains according to the presence or absence of the letter $0$ in $u^{(j)}$ with $j \in [k]$ by setting, for all $i \in [0,k]$,
\begin{equation}\label{equ:Zclassication}
{\mathcal{Z}}_i(n,k) := \{ [u^{(1)}, u^{(2)}, \dots, u^{(k)}] ~:~ 0 \in u^{(r)}, 0 \notin u^{(s)} \mbox{ for all } r \in [k-i], s \in [k-i+1,k] \}.
\end{equation}
This classification is called the \Def{${\mathcal{Z}}$-classification} for $k$-chains. Note that the union of all these sets is disjoint and give a description of all $k$-chains.
\smallbreak

For any $n \geq 2$, $k \geq 1$, $i \in [0,k]$, and $j \geq i$, let
\begin{equation}
\phi_i^{(n,k)} : {\mathcal{Z}}_i(n,k) \rightarrow \N \times {\mathcal{Z}}_j(n-1,k)
\end{equation}
such that, for $\gamma$ a $k$-chain in ${\mathcal{Z}}_i(n,k)$, 
\begin{equation}\label{equ:def-phi}
\phi_i^{(n,k)}(\gamma) := (\NumberTwo, \gamma') 
\end{equation}
where $\gamma'$ is the $k$-chain obtained by forgetting the last letter of each word of $\gamma$, and $\NumberTwo$ is the number of words ending by $2$ in $\gamma$.
\smallbreak

Let $\gamma \in {\mathcal{Z}}_i(n,k)$. Clearly, $\phi_i^{(n,k)}(\gamma)$ is a $k$-chain $\gamma'$ which belongs to ${\mathcal{Z}}_j(n-1,k)$ with $j \in [i, k]$, since the $k$-chain $\gamma'$ has at the most the same number of triwords with a letter $0$ than the $k$-chain $\gamma$.
\smallbreak

Therefore, by setting $\gamma := [v^{(1)}a^{(1)}, v^{(2)}a^{(2)}, \dots, v^{(k)}a^{(k)}]$ with $v^{(r)} \in~\SetTriword(n-1)$ and $a^{(r)} \in~[0,2]$ for all $r \in [k]$, and $(\NumberTwo, \gamma') := \phi_i^{(n,k)} (\gamma)$, there are two cases.
\smallbreak

\begin{itemize}
\item Suppose that $\gamma'$ belongs to ${\mathcal{Z}}_i(n-1,k)$. Then one has $k + 1$ possibilities to place or not the letter $2$.
Indeed, for $r \in [k-i]$, $a^{(r)} = 0$ or $a^{(r)} = 2$ because by hypothesis $0 \in v^{(r)}$. For $s \in [k - i + 1, k]$, because $\gamma'$ is already in ${\mathcal{Z}}_i(n-1,k)$, one has $a^{(s)} = 1$ or $a^{(s)} = 2$. To summarize, one has $k + 1$ possibilities to place the letter $2$, knowing that all letters before the first ending letter $2$ have to be smaller than $2$, and all letters after have to be $2$.

\item Suppose now that $\gamma'$ belongs to ${\mathcal{Z}}_j(n-1,k)$ with $j \in [i+1, k]$. Then one has $i+1$ possibilities to place or not the letter $2$. Indeed, in this case we must set $a^{(s)} = 0$ for all $s \in [k-j, k-i]$ in order to obtain a $k$-chain in ${\mathcal{Z}}_i(n,k)$. This implies that all ending letters before $a^{(k-j)}$ have to be also $0$. It follows that for all $r \in [k -i +1, k]$, $a^{(r)} = 1$ or $a^{(r)} = 2$. 
\end{itemize}
In the two cases, the position of the first letter $2$ depends on the integer $\NumberTwo$.
\smallbreak

Thus, for $\gamma$ a $k$-chain in ${\mathcal{Z}}_i(n,k)$, it follows that
\begin{equation}
\phi_i^{(n,k)}(\gamma) \in [k+1] \times {\mathcal{Z}}_i(n-1,k) ~~~~\bigsqcup~~~~ [i+1] \times \bigsqcup_{j \in [i+1, k]} {\mathcal{Z}}_j(n-1,k).
\end{equation}

For instance, by setting 
\begin{equation}\label{equ:gammachain}
\gamma := [00200, 02200, 02202, 12222]
\end{equation}
a $4$-chain of ${\mathcal{Z}}_1(5,4)$, one has $\phi_1^{(5,4)}(\gamma) = (\NumberTwo, \gamma')$ with 
\begin{equation}
\gamma' = [0020, 0220, 0220, 1222],
\end{equation}
and $\NumberTwo = 2$.
\smallbreak

\begin{Lemma}\label{lem:bijectionforkchain}
For any $n \geq 2$, $k \geq 1$, and $i \in [0,k]$, the map $\phi_i^{(n,k)}$ is a bijection.
\end{Lemma}

\begin{proof}
Let $\delta' := [v^{(1)}, v^{(2)}, \dots, v^{(k)}]$ be a $k$-chain of $\SetTriword(n-1)$, and $\NumberTwo \in [0, k]$.
\begin{itemize}
\item Suppose that $\delta' \in {\mathcal{Z}}_i(n-1,k)$.
Let $\delta := [v^{(1)}a^{(1)}, v^{(2)}a^{(2)}, \dots, v^{(k)}a^{(k)}]$ such that for all $r \in [k - \NumberTwo]$ we set $a^{(r)} = 0$ if $0 \in v^{(r)}$, and $a^{(r)} = 1$ otherwise, and $a^{(s)} = 2$ for all $s \in [k-\NumberTwo +1, k]$. The resulting $k$-chain is a $k$-chain of $\SetTriword(n)$ because $a^{(1)} \leq a^{(2)} \leq \dots \leq a^{(k)}$ by construction.
Furthermore, since no $0$ is added at the end of a word that does not contain a letter $0$ in $\delta'$, the $k$-chain $\delta$ belongs to ${\mathcal{Z}}_i(n,k)$.

\item Suppose that $\delta' \in {\mathcal{Z}}_j(n-1,k)$, with $j \in [i+1, k]$. Let $\delta := [v^{(1)}a^{(1)}, v^{(2)}a^{(2)}, \dots, v^{(k)}a^{(k)}]$ such that $a^{(r)} = 0$ for all $r \in [k-i]$, $a^{(s)} = 1$ for all $s \in [k-i+1, k-\NumberTwo]$, and $a^{(q)} = 2$ for all $q \in [k-\NumberTwo +1, k]$. By construction, one has $a^{(1)} \leq a^{(2)} \leq \dots \leq a^{(k)}$. This implies that this $k$-chain is a $k$-chain of $\SetTriword(n)$. 
Moreover, since the letter $0$ is added at the end of $v^{(r)}$ for $r \in [k-i]$, the $k$-chain $\delta$ belongs to ${\mathcal{Z}}_i(n,k)$.
\end{itemize}
In both cases, since $\delta$ belongs to ${\mathcal{Z}}_i(n,k)$, this implies that the map $\phi_i^{(n,k)}$ is surjective.
\smallbreak

Let $(\NumberTwo_1, \gamma')$ and $(\NumberTwo_2, \delta')$ be two pairs with $\NumberTwo_1, \NumberTwo_2 \in [0, k]$, and $\gamma' \in {\mathcal{Z}}_{j_1}(n-1,k)$ and $\delta' \in {\mathcal{Z}}_{j_2}(n-1,k)$ with $j_1, j_2 \in [i, k]$. Let $\gamma$ be the image of $(\NumberTwo_1, \gamma')$ and $\delta$ be the image of $(\NumberTwo_2, \delta')$ by ${\phi_i^{(n,k)}}^{-1}$. Suppose that $(\NumberTwo_1, \gamma') \ne (\NumberTwo_2, \delta')$. This implies that either $\NumberTwo_1 \ne \NumberTwo_2$ or $\gamma' \ne \delta'$. In the first case, if $\NumberTwo_1 > \NumberTwo_2$ then there are more words ending by $2$ in $\gamma$ than in $\delta$. Thus one has $\gamma \ne \delta$. In the second case, there is at least one word in $\gamma$ such that the prefix of this word is different from the word with the same index in $\delta$. Here again, one has $\gamma \ne \delta$. Hence, the map $\phi_i^{(n,k)}$ is injective.
\end{proof}

For instance, for the $4$-chain~\eqref{equ:gammachain}, $\gamma'$ belongs to ${\mathcal{Z}}_1(4,4)$ and $t$ is $2$. We can rebuild $\gamma$ by adding the letter $2$ on the two last words of $\gamma'$, since by definition of triwords, the greater triwords of a $k$-chain must have greater or equal letters compare to smaller triwords. Besides, since the two first words of $\gamma'$ have the letter $0$, we can only add the letter $0$ at its end.
\smallbreak

Let us consider another example with 
\begin{equation}
\gamma := [00000, 00200, 12210, 12211, 12212]
\end{equation}
a $5$-chain of ${\mathcal{Z}}_2(5,5)$. One has $\phi_2^{(5,5)}(\gamma) = (\NumberTwo, \gamma')$ with $\NumberTwo = 1$ and
\begin{equation}
\gamma' = [0000, 0020, 1221, 1221, 1221].
\end{equation}
Here $\gamma'$ belongs to ${\mathcal{Z}}_3(4,5)$. Since $\gamma \in {\mathcal{Z}}_2(5,5)$, to rebuild $\gamma$ from $\gamma'$, we have to add $0$ at the end of the third word of $\gamma'$. Moreover, since $\NumberTwo = 1$, the letter $2$ is added to the last word and the letter $1$ is added to the penultimate word of $\gamma'$.
\smallbreak

For any ${\mathcal{Z}}_i(n,k)$ of this classification, one obtains by denoting by ${\mathfrak{z}}_i(n,k)$ the cardinality of ${\mathcal{Z}}_i(n,k)$ with $i \in [0,k]$, the following result.
\smallbreak

\begin{Proposition}
Let $n \geq 2$ and $k \geq 1$. For all $i \in [0,k]$, each ${\mathfrak{z}}_i(n,k)$ satisfies
\begin{equation}\label{equ:z_i}
{\mathfrak{z}}_i(n,k) = (k+1) {\mathfrak{z}}_i(n-1,k) + (i+1) \sum_{j = i + 1}^{k} {\mathfrak{z}}_{j}(n-1,k).
\end{equation}
\end{Proposition}

\begin{proof}
This is a direct consequence of Lemma~\ref{lem:bijectionforkchain}.
\end{proof}

For example, for 
\begin{equation}
\begin{split}
{\mathcal{Z}}_1(2,3) =& \{[00,00,11], [00,00,12], [00,02,12], [02,02,12], \\
&[00,10,11], [00,10,12], [10,10,11], [10,10,12] \},
\end{split}
\end{equation}
the first four $3$-chains came from ${\mathcal{Z}}_1(1,3) = \{ [0,0,1] \}$, the next two came from ${\mathcal{Z}}_2(1,3) = \{ [0,1,1] \}$, and the last two came from ${\mathcal{Z}}_3(1,3) = \{ [1,1,1] \}$.
\smallbreak

The system
\begin{equation}
\begin{split}
{\mathfrak{z}}_0(n,k) &= (k+1) {\mathfrak{z}}_0(n-1,k) + {\mathfrak{z}}_1(n-1,k) + \dots + {\mathfrak{z}}_{k-1}(n-1,k) + {\mathfrak{z}}_{k}(n-1,k), \\
{\mathfrak{z}}_1(n,k) &= (k+1) {\mathfrak{z}}_1(n-1,k) + 2 {\mathfrak{z}}_2(n-1,k) + \dots + 2 {\mathfrak{z}}_{k-1}(n-1,k) + 2 {\mathfrak{z}}_{k}(n-1,k), \\
\vdots \\
{\mathfrak{z}}_{k-1}(n,k) &= (k+1) {\mathfrak{z}}_{k-1}(n-1,k) + k {\mathfrak{z}}_{k}(n-1,k), \\
{\mathfrak{z}}_k(n,k) &= (k+1) {\mathfrak{z}}_{k}(n-1,k),
\end{split}
\end{equation}
is called \Def{$\mathfrak{z}$-system}.

\begin{Proposition}\label{prop:k-chains}
For any $n \geq 2$ and $k \geq 1$, the $k$-chains of the Hochschild poset $\SetTriword(n)$ are enumerated by
\begin{equation}
\sum_{i = 0}^{k} {\mathfrak{z}}_i(n,k) = (k+1)^{n - (k + 1)} P_k(n),
\end{equation}
where $P_k(n)$ is a monic polynomial of degree $k$ determined by the $\mathfrak{z}$-system.
\end{Proposition}

\begin{proof}
Since for $n = 1$, all ${\mathfrak{z}}_i(1,k) = 1$ with $i \in [0, k]$, one can rewrite the $\mathfrak{z}$-system with matrices
\begin{equation}\label{equ:matrixM}
\begin{pmatrix}
{\mathfrak{z}}_0(n,k) \\
{\mathfrak{z}}_1(n,k) \\
\vdots \\
{\mathfrak{z}}_{k-1}(n,k) \\
{\mathfrak{z}}_k(n,k)
\end{pmatrix} =
\begin{pmatrix}
k+1 & 1 & 1 & \hdots & 1 \\
0 & k+1 & 2 & \hdots & 2 \\
\vdots & & \ddots & & \vdots \\
0 & \hdots & 0 &  k+1 & k \\
0 & \hdots & 0 & 0 & k+1
\end{pmatrix}^{n-1}
\begin{pmatrix}
1 \\
1 \\
\vdots \\
1 \\
1 
\end{pmatrix}.
\end{equation}

Let us denote by $M$ this upper triangular matrix, $I$ the identity matrix of dimension $k+1$, and $N := M -(k+1)I$. Since $I$ and $N$ commute, one has 
\begin{equation}
\begin{split}
M^{n-1} &= \Par{(k+1)I + N}^{n-1} \\
&= \sum_{i = 0}^{k} \binom{n-1}{i} (k+1)^{n-1-i} N^i \\
&= (k+1)^{n-(k+1)} \Par{(k+1)^{k} I + (n-1) (k+1)^{k-1} N + \dots + \frac{(n-1)!}{(n-k-1)! k!} N^k}\\
&= (k+1)^{n-(k+1)} Q_k(n),
\end{split}
\end{equation}
where $Q_k(n)$ is clearly polynomial in $n$.
It only remains to deduce the polynomial $P_k(n)$ from the matrix $Q_k(n)$, as the sum of all entries of $Q_k(n)$.
\smallbreak

Furthermore, $P_k(n)$ is a polynomial of degree $k$ since $n^k$ appears in $\dfrac{(n-1)!}{(n-k-1)! k!}$. 
\smallbreak

Moreover, a particular case from Lemma~\ref{lem:diagMatrix} gives that $N^k (1, k+1) = k!$. Since $N$ is a strictly upper triangular matrix, $N^k (1, k+1)$ is the only nonzero entry of $N^k$. This implies that $P_k(n)$ is a monic polynomial.
\end{proof}

\begin{Lemma}\label{lem:diagMatrix}
For any $n \geq 2$ and $k \geq 1$, let $M$ be the upper triangular matrix in~\eqref{equ:matrixM}, $I$ be the identity matrix of dimension $k+1$, and $N := M - (k+1)I$. For any $l \in [k]$ and $i \in [k+1]$ such that $i+l \leq k+1$, one has
\begin{equation}\label{equ:entryi,i+l}
N^l (i, i+l) = \frac{(i+l-1)!}{(i-1)!}.
\end{equation}
\end{Lemma}

\begin{proof}
We proceed by induction on $l$. Since $N(i,i+1) = i$ for all $i \in [k+1]$, one has that~\eqref{equ:entryi,i+l} follows for $l = 1$. Suppose that~\eqref{equ:entryi,i+l} is true for $l-1$ and let us consider $N^{l}$.
For any $i \in [k+1]$, one obtains $N^{l}(i,i+l)$ with the $i$-th line of $N^{l-1}$ and the $(i+l)$-th column of $N$. Since $N$ is a strictly upper triangular matrix, all left entries before $N^{l-1} (i, i+l-1)$ are zeros, and all below entries after $N(i+l-1,i+l)$ are also zeros. Therefore,
\begin{equation}
N^{l}(i,i+l) = N^{l-1} (i, i+l-1)~ N(i+l-1,i+l) = \frac{(i+l-2)!}{(i-1)!} (i+l-1) = \frac{(i+l-1)!}{(i-1)!},
\end{equation}
and then~\eqref{equ:entryi,i+l} holds for all $l \in [k]$.
\end{proof}

Note that since for $n = 1$, all ${\mathfrak{z}}_i(1,k) = 1$ with $i \in [0, k]$, the number of $k$-chains is $k+1$ for all $k \geq 1$. Using Proposition~\ref{prop:k-chains}, one can therefore deduce that $P_k(1) = (k+1)^{k+1}$.
\smallbreak

Recall that the triwords of size $n$ are enumerated by 
\begin{equation}
2^{n-2} (n + 3).
\end{equation}
A demonstration of this result is given in Section~\ref{subsec:Hoch}, involving generating series.  
By Proposition~\ref{prop:k-chains}, one has
\begin{equation}\label{equ:matrix-n=1}
\begin{pmatrix}
{\mathfrak{z}}_0(n,1) \\
{\mathfrak{z}}_1(n,1)
\end{pmatrix} =
\begin{pmatrix}
2 & 1 \\
0 & 2
\end{pmatrix}^{n-1}
\begin{pmatrix}
1 \\
1
\end{pmatrix} =
\begin{pmatrix}
2^{n-1} & (n-1)2^{n-2} \\
0 & 2^{n-1}
\end{pmatrix}
\begin{pmatrix}
1 \\
1
\end{pmatrix},
\end{equation}
which leads to the formula already known, for $n \geq 1$,
\begin{equation}
{\mathfrak{z}}_0(n,1) + {\mathfrak{z}}_1(n,1) = 2^{n-2} (n + 3).
\end{equation}
\smallbreak

Likewise, to enumerate the intervals of the Hochschild lattice, or in other words their $2$-chains, one has
\begin{equation}\label{matrix-n=2}
\begin{split}
\begin{pmatrix}
{\mathfrak{z}}_0(n,2) \\
{\mathfrak{z}}_1(n,2) \\
{\mathfrak{z}}_2(n,2)
\end{pmatrix} =&
\begin{pmatrix}
3 & 1 & 1  \\
0 & 3 & 2 \\
0 & 0 & 3 \\
\end{pmatrix}^{n-1}
\begin{pmatrix}
1 \\
1 \\
1 
\end{pmatrix} \\ 
=& \begin{pmatrix}
3^{n-1} & 3^{n-2}(n - 1) & 3^{n-2}(n -3) + 3^{n-3}(n^2-3n +8) \\
0 & 3^{n-1} & 3^{n-2}(2n -2)\\
0 & 0 & 3^{n-1} \\
\end{pmatrix}
\begin{pmatrix}
1 \\
1 \\
1 
\end{pmatrix}.
\end{split}
\end{equation}
The number of intervals of $\SetTriword(n)$ is therefore given by
\begin{equation}
{\mathfrak{z}}_0(n,2) + {\mathfrak{z}}_1(n,2) + {\mathfrak{z}}_2(n,2) = 3^{n-3} \Par{n^2 + 9n + 17}.
\end{equation}
\smallbreak

In the same way, the number of $3$-chains is
\begin{equation}
4^{n-4} \Par{n^3 + 20 n^2 + 93n + 142},
\end{equation}
the number of $4$-chains is
\begin{equation}
5^{n-5} \Par{n^4 + \frac{110}{3}n^3 + 355 n^2 + \frac{3490}{3} n + 1569},
\end{equation}
and the number of $5$-chains is
\begin{equation}
6^{n-6} \Par{n^5 + \frac{119}{2}n^4 + 1026 n^3 + \frac{13261}{2} n^2 + 17363 n + 21576}.
\end{equation}

It seems that the sequence of constant terms of the polynomials $P_k(n)$
\begin{equation}\label{equ:numbersconnectedfunctions}
3, 17, 142, 1569, 21576, \dots
\end{equation}
is the sequence of numbers of \Def{connected functions} on $n$ labeled nodes~\OEIS{A001865} of \cite{Slo}. Recall that a connected function is a function $f : [n] \rightarrow [n]$ such that the graph $G := (V, E)$ is connected, where $V := [n]$ is the set of vertices and $E := \{(i, f(i)) \}$ with $i \in [n]$ is the set of edges.

\subsection{Subposets of the Hochschild posets}

An interesting subposet of the poset $\SetTriword(n)$ appears by considering the set of triwords restricted to words beginning by the letter $1$. Here, some results are given for this subposet.
\smallbreak

Let $u \in \SetTriword(n)$ such that $u_1 = 1$, then $u$ is called a $\mu$-triword, and the graded set of $\mu$-triwords is denoted by $\SetTriword_\mu$. 
\smallbreak

From Lemma~\ref{grammar}, one has 
\begin{equation}
\SetTriword_\mu = \epsilon + 1B,
\end{equation}
where $B$ is the set of all words on $\{0,1,2\}$ avoiding the subword $01$.
\smallbreak

It follows that the generating series of $\SetTriword_\mu$ is 
\begin{equation}
G_{\SetTriword_\mu}(z) = 1 + zG_B(z).
\end{equation}
By reminding the two generating series~\eqref{ga} and~\eqref{gb}, one can deduce, for any $n \geq 1$,
\begin{equation}\label{numbermutriwords}
\# \SetTriword_\mu(n) = 2^{n-2}(n+1).
\end{equation}

The subposet $(\SetTriword_\mu(n), \Leq)$ is called \Def{mini-Hochschild poset}. As for Hochschild posets, we can give the ${\mathcal{Z}}$-classification for $k$-chains of mini-Hochschild posets.
This classification is identical to the classification~\eqref{equ:Zclassication}. For any $n \geq 2$, $k \geq 1$, and $i \in [0,k]$, let us show that the map $\phi_i^{(n,k)}$ defined by~\eqref{equ:def-phi} is also a bijection for the set of $\mu$-triwords. 
\smallbreak

First, the reverse image of the map $\phi_i^{(n,k)}$ adds one letter on the end of each triwords of the $k$-chains. It means that if all triwords of a $k$-chain $\gamma$ in $\mathcal{Z}_j (n-1,k)$ for $j \in [i,k]$ are $\mu$-triwords, then the reverse image of $\gamma$ is also a $k$-chain of $\mu$-triwords. Likewise, for a $k$-chain of $\mu$-triwords such that $\gamma \in \mathcal{Z}_i (n,k)$, $\phi_i^{(n,k)}(\gamma)$ remains a $k$-chain of $\mu$-triwords since the first letter of each $\mu$-triword remains $1$. Second, all arguments in the proof of Lemma~\ref{lem:bijectionforkchain} hold in the case of $\mu$-triwords because at no point the first letter of triwords which constitutes $k$-chains intervenes.
\smallbreak

The $\mathfrak{z}$-system for the mini-Hochschild poset holds, and one has for any $n \geq 2$, $k \geq 1$, and for all $i \in [0,k]$,
\begin{equation}
{\mathfrak{z}}_i(n,k) = (k+1) {\mathfrak{z}}_i(n-1,k) + (i+1) \sum_{j = i + 1}^{k} {\mathfrak{z}}_{j}(n-1,k).
\end{equation}
Since ${\mathfrak{z}}_k(1,k) = 1$ and ${\mathfrak{z}}_j(1,k) = 0$ for all $j \in [0, k-1]$, it follows that the $\mathfrak{z}$-system for the mini-Hochschild poset can be rewritten 
\begin{equation}
\begin{pmatrix}
{\mathfrak{z}}_0(n,k) \\
{\mathfrak{z}}_1(n,k) \\
\vdots \\
{\mathfrak{z}}_{k-1}(n,k) \\
{\mathfrak{z}}_k(n,k)
\end{pmatrix} =
\begin{pmatrix}
k+1 & 1 & 1 & \hdots & 1 \\
0 & k+1 & 2 & \hdots & 2 \\
\vdots & & \ddots & & \vdots \\
0 & \hdots & 0 &  k+1 & k \\
0 & \hdots & 0 & 0 & k+1
\end{pmatrix}^{n-1}
\begin{pmatrix}
0 \\
0 \\
\vdots \\
0 \\
1 
\end{pmatrix}.
\end{equation}

Thus, for any $n \geq 2$ and $k \geq 1$, the number of $k$-chains in the poset $\SetTriword_\mu(n)$ is given by the sum of the last column of $M^{n-1}$, where $M$ is the upper triangular matrix. One can conclude that Proposition~\ref{prop:k-chains} holds for the mini-Hochschild poset. 
\smallbreak

For instance, one deduce from~\eqref{equ:matrix-n=1} that the number of $\mu$-triwords of size $n$ is
\begin{equation}
2^{n-1} + (n-1)2^{n-2} = 2^{n-2}(n + 1),
\end{equation}
as shown through  generating series~\eqref{numbermutriwords}.
\smallbreak

In the same way, from~\eqref{matrix-n=2} one deduce that the number of intervals of $\SetTriword_\mu(n)$ is
\begin{equation}
3^{n-3}\Par{n^2 + 6n + 2},
\end{equation}
the number of $3$-chains is
\begin{equation}
4^{n-4}\Par{n^3 + 16n^2 + 41n + 6},
\end{equation}
the number of $4$-chains is
\begin{equation}
5^{n-5}\Par{n^4 + \frac{95}{3} n^3 + \frac{445}{2} n^2 + \frac{2075}{2} n + 24},
\end{equation}
and the number of $5$-chains is
\begin{equation}
6^{n-6}\Par{n^5 + \frac{107}{2} n^4 + 750 n^3 + \frac{6505}{2} n^2 + 3599 n + 120}.
\end{equation}

Similarly to the remark on the sequence of constant terms~\eqref{equ:numbersconnectedfunctions}, 
it seems that the sequence of constant terms of these polynomials
\begin{equation}
1, 2, 6, 24, 120, \dots
\end{equation}
is the sequence of factorial numbers.
\smallbreak

Several other properties verified by the Hochschild poset seem to hold for the mini-Hochschild poset. It may be interesting to proceed to a complete study of this subposet as well.

\section*{Appendix on Coxeter polynomials by Frédéric Chapoton}

This short section describes a conjectural property of the
Hochschild lattices, more precisely of their Coxeter polynomials.
\smallbreak

Let us start by a few general words on the Coxeter polynomial as an
interesting invariant of posets, and its theoretical context.
\smallbreak

Given any finite poset $\PosetP$, let $M_{\PosetP}$ be the triangular matrix with rows and columns indexed by $\PosetP$ and entries $M_{\PosetP}(x,y) = 1$ if $x \Leq_{\PosetP} y$ and $0$ otherwise. The Coxeter matrix of $\PosetP$ is the matrix $C_{\PosetP}$
defined by the formula $- M_{\PosetP} (M_{\PosetP}^{-1})^{t}$, where the second factor is the transpose of the inverse. This definition may look strange, but is very natural from a representation-theoretic point of view, where it comes from the Auslander-Reiten translation functor $\tau$ on the
derived category $\mathcal{D}_{\PosetP}$ of modules over the incidence algebra of $\PosetP$. To keep it short, let us just say that the Coxeter matrix $C_{\PosetP}$, up to change of basis over $\Z$, is an invariant of $\PosetP$ that depends only on the derived category $\mathcal{D}_{\PosetP}$. It is known that non-isomorphic posets can have equivalent derived categories, in which case they will share the same Coxeter matrix up to change of basis.
\smallbreak

The Coxeter polynomial of $\PosetP$, defined as the characteristic polynomial of the Coxeter matrix $C_{\PosetP}$, is therefore also an invariant of $\PosetP$ depending only on the derived category $\mathcal{D}_{\PosetP}$. This invariant is very easily computed on examples and sometimes turns out to have nice properties.
\smallbreak

In the case of Hochschild posets, computer experiments suggests the following conjecture.
\begin{conjecture} \label{conj_A}
The Coxeter polynomial $c_n(x)$ of the Hochschild poset $\SetTriword(n)$ is a product of cyclotomic polynomials.
\end{conjecture}
One can note that the Coxeter matrices for $\SetTriword(4)$ and $\SetTriword(5)$ are not diagonalizable over the complex numbers and do not have finite
multiplicative order.
\smallbreak

Moreover, one can propose a guess for the factorization, as
follows. Let $f_n$ be the Coxeter polynomial $c_n(x)$ of the
Hochschild poset $\SetTriword(n)$ if $n$ is odd and $(-1)^{\deg c_n} c_n(-x)$ if $n$ is
even.

\begin{conjecture}\label{conj_B}
The modified Coxeter polynomial $f_n$ can be written as
\begin{equation}
f_n(x) = \prod_{i \geq 1} (x^i - 1)^{d_n(i)},
\end{equation}
where the integers $d_n(i)$ have the description given below.
\end{conjecture}
Note that the description is to be taken as a first approximation only, as there are still ambiguities in the proposal for some exponents.
\smallbreak

Let us define integers $d_n(i)$ in two steps. Unless defined below, $d_n(i) = 0$.
\smallbreak

First one easy step:
\begin{itemize}
\item for $i = 1$, $d_n(1) = (-1)^n$,
\item for $i = 3$, $d_n(3) = 1 - \dfrac{2^{n-1}+(-1)^n}{3}$,
\item for $i$ a multiple of $3$ with $3 < i \leq n + 2$, $d_n(i) = \binom{n - 1}{i - 3}$.
\end{itemize}

Then comes the second step, which is more complicated.
For every integer $k$ with $1 \leq k \leq (n+1)/3$, let
\begin{equation}
I_k := (3 k + 2) n - 3 k + 1
\quad\text{and}\quad
D_k := \binom{n - 1}{3 k - 2}/{(3 k - 1)}.
\end{equation}

Thus,
\begin{itemize}
\item if $D_k$ is an integer, then one sets $d_n(I_k) := D_k$,
\item otherwise, one sets $d_n(I_k) := \lfloor D_k \rfloor$ and $d_n(I_k (D_k- \lfloor D_k \rfloor)) = 1$.
\end{itemize}

This finishes the proposed description for the exponents $d_n(i)$. The last
case is the ambiguous place, as the known values
were not sufficient to make a better guess for splitting
$I_k (D_k- \lfloor D_k \rfloor)$ into the product of an index and an
exponent.

\medskip

Table~\ref{tab:coxterpoly} depicts the known values for the modified Coxeter polynomials $f_n(x)$, where we abbreviate
$(x^n - 1)^k$ as $n^k$. In each case, one can check that the proposed description for the exponents does work.
\begin{table}\renewcommand{\arraystretch}{2.3}\footnotesize
  \begin{tabular}{|c|c|c|}
    \hline
    $n$ & $\# \SetTriword(n)$ & $f_n$\\
                 \hline
  $1$ & $2$ & $\dfrac{3}{1}$\\
  $2$ & $5$ & $\dfrac{1 \cdot 8}{4}$\\
  $3$ & $12$ & $\dfrac{13}{1}$\\
  $4$ & $28$ & $\dfrac{1\cdot  6 \cdot 9 \cdot 18}{3^2}$\\
  $5$ & $64$ & $\dfrac{6^4 \cdot 7 \cdot 23^2}{1 \cdot 3^4}$\\
  $6$ & $144$ & $\dfrac{1 \cdot 6^{10} \cdot 28^3 \cdot 43}{3^{10} \cdot 14}$\\
  $7$ & $320$ & $\dfrac{6^{20} \cdot 9 \cdot 33^3\cdot 51^3}{1\cdot 3^{20}}$\\
  $8$ & $704$ & $\dfrac{1\cdot 6^{35}\cdot 9^{7}\cdot 19\cdot 20\cdot 38^3\cdot 59^7}{3^{42}\cdot 10}$\\
  $9$ & $1536$ & $\dfrac{6^{56}\cdot 9^{28}\cdot 43^{4}\cdot 67^{14}\cdot 91}{1\cdot 3^{84}}$\\
  $10$ & $3328$ & $\dfrac{1\cdot 6^{84}\cdot 9^{84}\cdot 12\cdot 15\cdot 48^5\cdot 75^{25}\cdot 102^5}{3^{170}\cdot 24\cdot 51}$\\
    \hline
\end{tabular}
\smallbreak
\caption{\footnotesize Some values for the modified Coxeter polynomials $f_n(x)$.}
    \label{tab:coxterpoly}
\end{table}

\newpage
\bibliographystyle{alpha}
\bibliography{Bibliography}

\begin{thebibliography}{Com19}

\bibitem[Bir37]{Bir37}
G.~Birkhoff.
\newblock Rings of sets.
\newblock {\em Duke Math. J.}, 3(3):443--454, 1937.

\bibitem[BS97]{BS97}
A.~Blass and B.~Sagan.
\newblock Möbius functions of lattices.
\newblock {\em Adv. in Math.}, 127:94--123, 1997.

\bibitem[BW96]{BW96}
A.~Bj\"{o}rner and M.~L. Wachs.
\newblock Shellable nonpure complexes and posets. {I}.
\newblock {\em Trans. Amer. Math. Soc.}, 348(4):1299--1327, 1996.

\bibitem[BW97]{BW97}
A.~Bj\"{o}rner and M.~L. Wachs.
\newblock Shellable nonpure complexes and posets. {II}.
\newblock {\em Trans. Amer. Math. Soc.}, 349(10):3945--3975, 1997.

\bibitem[CG20]{CG20}
C.~Combe and S.~Giraudo.
\newblock {Three interacting families of Fuss-Catalan posets}.
\newblock {\em Formal Power Series and Algebraic combinatorics}, 2020.

\bibitem[Cha20]{Cha20}
F.~Chapoton.
\newblock Some properties of a new partial order on {D}yck paths.
\newblock {\em Algebraic Combinatorics}, 3:433--463, 2020.

\bibitem[Com19]{Com19}
C.~Combe.
\newblock {Cubic realizations of Tamari interval lattices}.
\newblock {\em S\'em. Lothar. Combin.}, 82B:Article $\# 23$, 2019.

\bibitem[Day79]{Day79}
A.~Day.
\newblock Characterizations of finite lattices that are bounded-homomorphic
  images of sublattices of free lattices.
\newblock {\em Canadian J. Math.}, 31(1):69--78, 1979.

\bibitem[Mar92]{Mar92}
G.~Markowsky.
\newblock Primes, irreducibles and extremal lattices.
\newblock {\em Order}, 9(3):265--290, 1992.

\bibitem[M{\"{u}}h19]{Muh19}
H.~M{\"{u}}hle.
\newblock The core label order of a congruence-uniform lattice.
\newblock {\em Algebra Universalis}, 80(1):Art. 10, 22, 2019.

\bibitem[RS18]{RS18}
M.~Rivera and S.~Saneblidze.
\newblock A combinatorial model for the free loop fibration.
\newblock {\em Bull. Lond. Math. Soc.}, 50(6):1085--1101, 2018.

\bibitem[Sak09]{Sak09}
J.~Sakarovitch.
\newblock {\em Elements of automata theory}.
\newblock Cambridge University Press, Cambridge, 2009.
\newblock Translated from the 2003 French original by Reuben Thomas.

\bibitem[San09]{San09}
S.~Saneblidze.
\newblock The bitwisted {C}artesian model for the free loop fibration.
\newblock {\em Topology Appl.}, 156(5):897--910, 2009.

\bibitem[San11]{San11}
S.~Saneblidze.
\newblock On the homology theory of the closed geodesic problem.
\newblock {\em Rep. Enlarged Sess. Semin. I. Vekua Appl. Math.}, 25:113--116,
  2011.

\bibitem[Slo]{Slo}
N.~J.~A. Sloane.
\newblock {The On-Line Encyclopedia of Integer Sequences}.
\newblock \url{https://oeis.org/}.

\bibitem[Sta11]{Sta11}
R.~P. Stanley.
\newblock {\em {Enumerative Combinatorics}}, volume~1.
\newblock Cambridge University Press, second edition, 2011.

\bibitem[Tho06]{Tho06}
H.~Thomas.
\newblock {An analogue of distributivity for ungraded lattices}.
\newblock {\em Order}, 23(2-3):249--269, 2006.

\bibitem[TW19]{TW19}
H.~Thomas and N.~Williams.
\newblock Rowmotion in slow motion.
\newblock {\em Proc. Lond. Math. Soc. (3)}, 119(5):1149--1178, 2019.

\end{thebibliography}

\end{document}